\documentclass[11pt, reqno]{amsart}
   \usepackage{amsmath}
   \usepackage{amsthm}
   \usepackage{amsfonts}
   \usepackage{stmaryrd}
   \usepackage{amssymb}
   \usepackage{pgfplots}
   \usepackage{tikz-cd}
   \usepackage{braket}
   \usepackage{indentfirst}
   \usepackage[a4paper, margin=0.9in]{geometry}
   \usepackage{bm}
   \usepackage{scalerel}
   \usepackage{enumerate}
\usepackage{stackengine,wasysym}
\usepackage{faktor}
\usepackage{adjustbox}
\usepackage{mathtools}
\usepackage{chngcntr}
\usepackage{hyperref}
\usepackage{enumitem}

\usepackage[utf8]{inputenc}

\usepackage[intoc, refpage]{nomencl}
  \newcommand{\twodigitspage}{
   \ifnum \value{page} < 10
   0
   \fi
   \thepage
  }
  \tikzset{
    slanted/.style={rotate=-90, anchor = south},
    slantedswap/.style={rotate=-90, anchor=north, outer sep=0.75mm}
}

   \theoremstyle{definition}
   \newtheorem{definition}{Definition}[section]
   \newtheorem{remark}[definition]{Remark}
   \newtheorem{example}[definition]{Example}
   \newtheorem*{acknowledgments}{Acknowledgements}
   
   \theoremstyle{plain}   
   \newtheorem{proposition}[definition]{Proposition}
   \newtheorem{lemma}[definition]{Lemma}
   \newtheorem{theorem}[definition]{Theorem}
   \newtheorem{corollary}[definition]{Corollary}

   \newtheorem{hypothesis}[definition]{Hypothesis}
   \newtheorem{conjecture}[definition]{Conjecture}
   \newtheorem{question}[definition]{Question}
   
   \newtheorem{aproposition}{Proposition}[subsection]
   \newtheorem{ahypothesis}[aproposition]{Hypothesis}
   \newtheorem{aexample}[aproposition]{Example}
   \newtheorem{alemma}[aproposition]{Lemma}
   
   \newtheorem{aquestion}[aproposition]{Question}
   
   \theoremstyle{definition}
   \newtheorem{aremark}[aproposition]{Remark}

   \newcommand{\ur}{{\operatorname{ur}}}
   
\newcommand{\Hom}{\mathrm{Hom}}

\newcommand{\Aut}{\operatorname{Aut}}
\newcommand{\Gal}{\operatorname{Gal}}
\newcommand{\Ind}{\operatorname{Ind}}

\newcommand{\Spec}{\operatorname{Spec}}
\newcommand{\ram}{\mathrm{ram}}

\newcommand{\Ext}{{\operatorname{Ext}}}
\newcommand{\catname}[1]{\textnormal{{\textsf{#1}}}}

\newcommand{\Der}{\catname{D}}
\newcommand{\R}{\catname{R}}

\renewcommand{\Ext}{\catname{Ext}}

\newcommand{\Det}{\catname{d}}

\newcommand{\Cl}{\operatorname{Cl}}
\newcommand{\Iw}{\mathrm{Iw}}
\newcommand{\tor}{\mathrm{tor}}

\newcommand{\bidual}{\bigcap\nolimits}
\newcommand{\exprod}{\bigwedge\nolimits}

\newcommand{\ES}{\operatorname{ES}}
\newcommand{\VS}{\operatorname{VS}}
\newcommand{\HS}{\operatorname{HS}}
\newcommand{\Ann}{\operatorname{Ann}}
\newcommand{\Fitt}{\operatorname{Fitt}}

\newcommand{\ab}{\mathrm{ab}}
\newcommand{\cyc}{\mathrm{cyc}}

   \newcommand{\bLambda}{{\mathpalette\makebLambda\relax}}
\newcommand{\makebLambda}[2]{%
  \raisebox{\depth}{\scalebox{1}[-1]{$\mathsurround=0pt#1\mathbb{V}$}}%
}

 \makeatletter
\newcommand{\colim@}[2]{%
  \vtop{\m@th\ialign{##\cr
    \hfil$#1\operator@font colim$\hfil\cr
    \noalign{\nointerlineskip\kern1.5\ex@}#2\cr
    \noalign{\nointerlineskip\kern-\ex@}\cr}}%
}
\newcommand{\colim}{%
  \mathop{\mathpalette\colim@{\rightarrowfill@\scriptscriptstyle}}\nmlimits@
}
\renewcommand{\varprojlim}{%
  \mathop{\mathpalette\varlim@{\leftarrowfill@\scriptscriptstyle}}\nmlimits@
}
\renewcommand{\varinjlim}{%
  \mathop{\mathpalette\varlim@{\rightarrowfill@\scriptscriptstyle}}\nmlimits@
}
\makeatother

  \DeclareMathOperator*{\im}{im}

  \newcommand*{\cores}{\mathrm{cores}}

  \DeclareMathOperator{\Fr}{\mathrm{Fr}}

\makeatletter
\newcommand{\varprojRlim@}[2]{
  \vtop{\m@th\ialign{##\cr
    \hfil$#1\catname{R}\operator@font lim$\hfil\cr
    \noalign{\nointerlineskip\kern1.5\ex@}#2\cr
    \noalign{\nointerlineskip\kern-\ex@}\cr}}%
}
\newcommand{\varprojRlim}{%
  \mathop{\mathpalette\varprojRlim@{\leftarrowfill@\textstyle}}\nmlimits@
}
\makeatother

\newcommand{\QQ}{\mathbb{Q}}
\newcommand{\ZZ}{\mathbb{Z}}

\newcommand{\NN}{\mathbb{N}}

\newcommand{\kk}{\Bbbk}
\newcommand{\cO}{\mathcal{O}}
\newcommand{\cL}{\mathcal{L}}

\newcommand{\cK}{\mathcal{K}}
\newcommand{\cG}{\mathcal{G}}
\newcommand{\cB}{\mathcal{B}}
\newcommand{\cR}{\mathcal{R}}
\newcommand{\cA}{\mathcal{A}}
\newcommand{\cP}{\mathcal{P}}
\newcommand{\cF}{\mathcal{F}}
\newcommand{\cT}{\mathcal{T}}
\newcommand{\cD}{\mathcal{D}}
\newcommand{\cN}{\mathcal{N}}
\newcommand{\cH}{\mathcal{H}}

\newcommand{\fr}{\mathfrak{r}}
\newcommand{\fn}{\mathfrak{n}}

\newcommand{\fq}{\mathfrak{q}}
\newcommand{\fp}{\mathfrak{p}}

\newcommand{\cJ}{\mathcal{J}}

\newcommand{\RES}{\operatorname{RES}}
\newcommand{\KS}{\operatorname{KS}}
\newcommand{\StS}{\operatorname{SS}}

\makeatletter
\def\th@plain{%
  \thm@notefont{}
  \itshape 
}
\def\th@definition{%
  \thm@notefont{}
  \normalfont 
}
\makeatother

\begin{document}
\author{Alexandre Daoud}
\date{}
\title[]{On the structure of the module of Euler systems \\ for a $p$-adic representation}

\address{King's College London,
Department of Mathematics,
London WC2R 2LS,
U.K.}
\email{alexandre.daoud@kcl.ac.uk}

\vspace*{-1cm}

\maketitle

\begin{abstract}
    We investigate a question of Burns and Sano concerning the structure of the module of Euler systems for a general $p$-adic representation. Assuming the weak Leopoldt conjecture, and the vanishing of $\mu$-invariants of natural Iwasawa modules, we obtain an Iwasawa-theoretic classification criterion for Euler systems which can be used to study this module. This criterion, taken together with Coleman's conjecture on circular distributions, leads us to pose a refinement of the aforementioned question for which we provide strong, and unconditional, evidence. We furthermore answer this question in the affirmative in many interesting cases in the setting of the multiplicative group over number fields. As a consequence of these results, we derive explicit descriptions of the structure of the full collection of Euler systems for the situations in consideration.
\end{abstract}

\let\thefootnote\relax

\tableofcontents

\section{Introduction}

The theory of Euler systems, and thus also the closely related theory of Kolyvagin systems, plays an important role in determining the structure of Selmer groups attached to particular $p$-adic representations. The existence of an appropriate non-trivial Euler system in a given context has applications to both the Bloch-Kato Conjecture and Iwasawa Main Conjectures; as was demonstrated, for example, by Kolyvagin in his seminal article \cite{kolyvagin} and Rubin in \cite{rubin2}.

It seems to be the case, however, that the problem of constructing an Euler system for a given representation is extremely difficult. To date only a handful of families of $p$-adic representations have been discovered to possess a non-trivial Euler system related to $L$-values. Some examples have arisen via classical and well-understood algebraic number theory such as the system of cyclotomic units. Others were discovered only after proving deep results in arithmetic geometry. As an example, one has Kato's Euler system for the Tate module of an elliptic curve over $\QQ$; its existence effectively depending on the celebrated modularity theorem.

To further complicate matters, there exists a vast class of representations that do not seem to fit into the framework given by the definition of an Euler system as a system of cohomology classes satisfying corestriction relations. In her work \cite{perrin-riou}, Perrin-Riou attempted to rectify this by asking for Euler systems to be families of elements of appropriate exterior powers of cohomology modules: so-called `higher rank' Euler systems. While the incorporation of a notion of `rank' into the theory of Euler systems has proven malleable, it still seems a daunting task to (unconditionally) construct examples of such a system. 

Motivated by the Rubin-Stark conjecture, Burns and Sano slightly modified in \cite{bs} the definition of a higher rank Euler system to be a family of elements in so-called `exterior biduals' of appropriate cohomology groups. Armed with this definition, they were able to construct a module $\ES^b(T)$ of higher rank Euler systems which is non-zero under mild hypotheses on the given representation $T$. They then went on in \cite[Rem. 1.6]{bs} to explicitly ask how common it is for an arbitrary element of the full module of Euler systems $\ES(T)$ to be an element of $\ES^b(T)$.

In this present article we attempt to provide a conjectural framework, and some results, towards this question. For suitably chosen representations we pose in Question \ref{main-question} the following refinement of the question of Burns and Sano: are \textit{all} systems in $\ES(T)$ contained in $\ES^b(T)$? This question can be seen as an analogue for $p$-adic representations of \cite[Conj. 2.5]{bdss} formulated by Burns, Sano, Seo and the presently named author. The formulation of Question \ref{main-question} is thus motivated by Coleman's conjecture on circular distributions which is equivalent to the assertion that, modulo torsion, all Euler systems defined over abelian extensions of $\QQ$ arise from the cyclotomic Euler system.

An affirmative answer to our question would provide a heuristic for why Euler systems tend to be so difficult to construct. Indeed, Question \ref{main-question} lends credence to the belief that for nice enough representations (such as those coming from geometry), all its associated Euler systems of appropriate rank arise from determinants of particular cohomology complexes $C_{E,S(E)}$ (see Appendix \ref{cohomology-appendix}). While these determinants tend to be algebraically supple they are very mysterious objects in general; for example, they appear as central players in the Burns-Flach formulation of the equivariant Tamagawa Number Conjecture.

In the course of this article we shall use the equivariant theory of higher rank Kolyvagin and Stark systems of Burns, Sakamoto and Sano to provide strong unconditional evidence for Question \ref{main-question} for representations and fields satisfying so-called `standard hypotheses'. We will then go on to employ Iwasawa theory to give an explicit criterion for an Euler system to be an element of $\ES^b(T)$ in terms of classical Iwasawa-theoretic divisibilities. 

We moreover leverage the classical theory of Euler systems in the sense of Kolyvagin to give a proof of Question \ref{main-question} for particular twisted representations of $\ZZ_p(1)$ over $\QQ$. Consequently we are able to obtain an explicit description of the Galois module structure of the full module of Euler systems for these representations. To be more precise, we show that the module of rank one Euler systems for $\QQ$ and the aforementioned twisted representations is free of rank one over the appropriate Iwasawa algebra with a basis given by the Euler system of cyclotomic units.

While we are not at this stage able to give an answer to Question \ref{main-question} for more general representations over number fields (due, in part, to the difficult condition of vanishing $\mu$-invariants), we are able to give partial `systems-level' results. For example, we show that for any elliptic curve over $\QQ$ satisfying particular technical hypotheses, there exists a basic Euler system which agrees with Kato's Euler system on the cyclotomic $\ZZ_p$-extension of $\QQ$.

\subsection{Notation and set-up}

\subsubsection{Arithmetic}

Throughout this article we fix a rational \textit{odd} prime $p$ and a number field $K$. We denote by $\overline{K}$ a fixed algebraic closure and $G_K$ its absolute Galois group. We fix also a finite extension $Q$ of $\QQ_p$ and write $R$ for its ring of integers. We assume to be given a finite-dimensional separable commutative $Q$-algebra $\cB$ and a Gorenstein $R$-order $\cR$ in $\cB$. We denote by $T$ a $p$-adic representation of $K$ with coefficients in $\cR$ (that is to say, a discrete free $\cR$-module admitting a continuous action of $G_K$).

Given any abelian extension $\mathcal{K}$ of $K$ we write $\Omega(\mathcal{K}/K)$ for the collection of finite intermediate fields of $\mathcal{K}/K$. For a fixed $E \in \Omega(\mathcal{K}/K)$ we write $\cG_E$ for its Galois group $\Gal(E/K)$, $S_\infty(E)$ for the set of its archimedean places and $S_\ram(E/K)$ for those primes of $K$ that ramify in $E$. Moreover, given any set of places $S$ of $K$, we write $S(E) := S \cup S_\ram(E/K)$ and $S_E$ for the set of those places of $E$ that lie above those in $S$. Except in cases of ambiguity we will often omit this subscript.

We write $S_p(K)$ for the set of $p$-adic places of $K$ and $S_\ram(T)$ for the set of places of $K$ at which $T$ is ramified. We then denote
\begin{align*}
    S_\mathrm{min}(T) := S_\infty(K) \cup S_p(K) \cup S_\ram(T)
\end{align*}
and fix forever a set $S$ of places of $K$ containing $S_\mathrm{min}(T)$.

For each prime $v \not\in S$ we fix a Frobenius automorphism $\Fr_v$ and define a polynomial
\begin{align*}
    P_v(x) := \det(1-\Fr_v^{-1}\mid T) \in \cR[x]
\end{align*}

For every positive integer $m$ we denote by $\mu_m$ the group of $m^{th}$ roots of unity in $\overline{\QQ}^\times$. We then define a $p$-adic representation $\ZZ_p(1) := \varprojlim_m \mu_m$ of $G_K$, set $\cR(1) := \cR \otimes_{\ZZ_p} \ZZ_p(1)$ and define the Kummer dual $T^*(1) := \Hom_\cR(T, \cR(1))$.

We denote by $T_E$ the induced module $\Ind_{G_K}^{G_E}(T)$. We remark that $T_E$ identifies with $\cR[\cG_E] \otimes_{\cR} T$ upon which $G_K$ acts via the rule
\begin{align*}
    \sigma\cdot(a \otimes t) := a\overline{\sigma}^{-1} \otimes \sigma t
\end{align*}
for $\sigma \in G_K, a \in \cR[\cG_E]$ and $t \in T$ where $\overline{\sigma}$ is the image of $\sigma$ in $\cG_E$.
\subsubsection{Homological Algebra}

For any commutative ring $A$ we denote by $\Der(A)$ the derived category of complexes of $A$-modules and by $\Der^p(A)$ the full triangulated subcategory of $\Der(A)$ consisting of those complexes which are perfect (in other words, those complexes which are isomorphic in $\Der(A)$ to a bounded complex of finitely generated projective $A$-modules).

Given a Galois extension of fields $F/E$ and $N$ an $A[\cG_E]$-module admitting a continuous $\Gal(F/E)$-action, we write $\R\Gamma(F/E, N)$ for the complex of continuous cochains of $N$. If $F$ is the separable closure of $E$ then we abbreviate this to just $\R\Gamma(E,N)$.

For any finite set of places $U$ of $K$ containing $S_\mathrm{min}(N)$, let $E_U$ denote the maximal Galois extension of $E$ unramified outside of the places lying above those in $U$. Then $N$ is naturally a $\Gal(E_U/E)$ module and we write $\R\Gamma(\cO_{E,U}, N)$ for $\R\Gamma(E_U/E,N)$.

We define the compactly-supported cohomology complex $\R\Gamma_c(\cO_{E,U}, N)$ to be the mapping fibre in $\Der(A[\cG_E])$ of the natural localisation morphism
\begin{align*}
    \R\Gamma(\cO_{E,U}, N) \to \bigoplus_{w \in U_E} \R\Gamma(E_w, N)
\end{align*}
induced by inflation with respect to $G_E \to \Gal(E_U/E)$ and then restriction with respect to $G_{E_w} \subseteq G_E$.

Now fix a place $w \not\in U_E$. We define the finitely-supported cohomology complex
\begin{align*}
    \R\Gamma_f(E_w,N) := \R\Gamma(E_w^\ur/E_w, N)
\end{align*}
where $E_w^\ur$ is the maximal unramified extension of $E_w$. We remark that this complex is isomorphic in $\Der^p(A[\cG_E])$ to the complex
\begin{align}\label{finite-support-resolution}
    [N \xrightarrow{1-\Fr_w^{-1}} N]
\end{align}
concentrated in degrees zero and one.

For each $* \in \set{\varnothing, c, f}$ and $i \in \ZZ$, we denote by $H^i_*(-,N)$ the cohomology modules $H^i(\R\Gamma_*(-,N))$.

\subsubsection{Algebra}

Let $X$ be an $A$-module. We write $X^* := \Hom_A(X,A)$. We write $X_\tor$ for its torsion submodule as an abelian group and $X_\mathrm{tf}$ for the quotient $X/X_\tor$. If, moreover, $A$ is an integral domain with fraction field $F$ then we write $A^\lor$ for the Pontryagin dual $\Hom_A(X,F/A)$.

We write $\Det_A$ for the determinant functor of Knudsen and Mumford from the category of finitely generated projective $A$-modules to that of graded invertible $A$-modules. We recall that this functor induces a functor on $\Der^p(A)$ which, by abuse of notation, we also denote $\Det_A$.

When $G$ is an abelian group, we denote by $(-)^\#$ the involution functor on the category of $A[G]$-modules which sends an $A[G]$-module $X$ to the module with the same underlying abelian group structure but with the $G$-action twisted by the involution $g \mapsto g^{-1}$ of $A[G]$.

Given an integer $r \geq 0$ we define the $r^{th}$ exterior bidual of $X$ to be the module
\begin{align*}
    \bidual_A^r X := \left(\exprod_A^r X^*\right)^*
\end{align*}
We remark that there is a natural homomorphism
\begin{align*}
    \xi_r: \exprod_A^r X &\to \bidual_A^r X\\
    x &\mapsto (\phi \mapsto \phi(x))
\end{align*}
which is, in general, neither injective nor surjective. We also recall from \cite[Prop. A.8]{bs} that when $A = \cR[\cG_E]$ for some $E \in \Omega(\cK/K)$, then $\xi_r$ induces an isomorphism
\begin{align*}
    \Set{x \in Q \otimes_R \exprod_{\cR[\cG_E]}^r X | \Phi(x) \in \cR[\cG_E] \text{ for all } \Phi \in \exprod_{\cR[\cG_E]}^r X^*} \xrightarrow{\sim} \bidual_{\cR[\cG_E]}^r X
\end{align*}

If $G$ is finite, $X$ a $\ZZ[G]$-module, and $\psi \in \widehat{G}$ a character, we define the $(p,\psi)$-part $X^\psi$ of $X$ to be the $\ZZ_p[\im(\psi)]$-module $X \otimes_{\ZZ[G]} \ZZ_p[\im(\psi)]$ where we regard $\ZZ_p[\im(\psi)]$ as a $\ZZ_p[G]$-algebra via $\psi$. Given $x \in X$ we write $x^\psi$ for the element $x \otimes 1$ of $X^\psi$.

\subsection{Statement of the question and main results}

\subsubsection{Modules of Euler systems}

In this subsection we briefly review the various modules of Euler systems with which we shall be concerned. To do this, we fix an abelian extension $\mathcal{K}/K$.

Given a pair of fields $E \subseteq E'$ in $\Omega(\cK/K)$ we note that the corestriction map induces, for each $r \geq 0$, a morphism
\begin{align*}
    \cores_{E'/E}: Q \otimes_R \exprod_{\cR[\cG_{E'}]}^r H^1(\cO_{E',S(E')}, T) \to Q \otimes_R \exprod_{\cR[\cG_E]}^r H^1(\cO_{E,S(E')}, T)
\end{align*}
With this in mind we then make the following definition:

\begin{definition}
    A rational Euler system of rank $r \geq 0$ for the pair $(T,\cK)$ is a collection
    \begin{align*}
        \eta := (\eta_E)_E \in \prod_{E \in \Omega(\cK/K)} Q \otimes_R \exprod_{\cR[\cG_E]}^r H^1(\cO_{E,S(E)}, T)
    \end{align*}
    such that for every pair of fields $E \subseteq E'$ in $\Omega(\cK/K)$ one has
    \begin{align*}
        \cores_{E'/E}(\eta_{E'}) = \left(\prod_{v \in S(E')\backslash S(E)} P_v(\Fr_v^{-1})\right)\eta_E
    \end{align*}
    inside $Q \otimes_R \exprod_{\cR[\cG_E]}^r H^1(\cO_{E,S(E')}, T)$. We denote by $\RES_r(T,\cK)$ the collection of such rational Euler systems of rank $r$ for $(T,\cK)$ and note that it can naturally be endowed with the structure of an $\cR[[\Gal(\cK/K)]]$-module.
\end{definition}

\begin{definition}\label{euler-system-definition}
    An Euler system of rank $r \geq 0$ for the pair $(T,K)$ is an element $\eta \in \RES_r(T,\cK)$ with the property that for each $E \in \Omega(\cK/K)$ one has that
    \begin{align*}
        \eta_E \in \bidual_{\cR[\cG_E]}^r H^1(\cO_{E,S(E)}, T)
    \end{align*}
    The set of all such Euler systems of rank $r$ constitutes an $\cR[[\Gal(\cK/K)]]$-submodule of $\RES_r(T,\cK)$ which we denote by $\ES_r(T,\cK)$.
\end{definition}

\subsubsection{Admissible pairs}

We consider the following natural hypotheses on the pair $(T,\cK)$:

\begin{hypothesis}\label{T-hypothesis}\text{}
    \begin{enumerate}
        \item[$\mathrm{(H_0)}$]{The $\cR$-module $Y_K(T) := \bigoplus_{v \in S_\infty(K)} H^0(K_v,T^*(1))$ is free of rank $ r \geq 1$.}
        \item[$(\mathrm{H_1})$]{$T$ is ramified at only finitely many primes of $K$.}
        \item[$(\mathrm{H_2})$]{For all $E \in \Omega(\cK/K)$ the invariants module $H^0(E, T)$ vanishes.}
        \item[$(\mathrm{H_3})$]{For all $E \in \Omega(\cK/K)$ the cohomology group $H^1(\cO_{E,S(E)}, T)$ is $R$-torsion-free.}
    \end{enumerate}
\end{hypothesis}

\begin{hypothesis}\label{k-hypothesis}\text{}$\cK$ is an abelian pro-$p$ extension such that
    \begin{enumerate}
        \item{$\cK$ contains the maximal $p$-extension inside the ray class field modulo $\mathfrak{q}$ for almost all primes $\mathfrak{q}$ of $K$;}
        \item{$\cK = \cK'K_\infty$ where $\cK'$ is unramified at all places in $S$ and $K_\infty$ is a $\ZZ_p$-extension of $K$, with Galois group $\Gamma$, which is disjoint from the Hilbert class field of $K$ and in which no non-archimedean place splits completely.}
    \end{enumerate}
\end{hypothesis}

\begin{remark}
    Assume $\cK$ is chosen to satisfy Hypothesis \ref{k-hypothesis} and fix a field $E \in \Omega(\cK'/K)$. Then it is easy to check that there is a canonical decomposition of Galois groups
    \begin{align*}
        \Gal(EK_\infty/K) \cong \Gal(E/K) \times \Gal(K_\infty/K)
    \end{align*}
    To give an example of when this hypothesis is satisfied, one can take $p$ to be a prime not dividing the class number of $K$, $K_\infty$ a $\ZZ_p$-extension of $K$ and $\cK'$ the maximal abelian extension of $K$ which is unramified at all primes in $S$.
\end{remark}

Throughout this article we assume that all pairs $(T,\cK)$, unless otherwise explicitly stated, satisfy these hypotheses. In this case we shall refer to such pairs as `admissible pairs'.

\begin{example}
    If we set $T = \ZZ_p(1)$, then Kummer Theory provides us with an identification
    \begin{align*}
        H^1(\cO_{E,S(E)}, T) \cong \cO_{E,S(E)}^\times \otimes_\ZZ \ZZ_p
    \end{align*}
    for each $E \in \Omega(\cK/K)$. In particular, if $K$ (and therefore also $E$) is totally real then each $H^1(\cO_{E,S(E)}, T)$ is $\ZZ_p$-torsion-free (since $p$ is assumed odd). In this situation one then has that $(T,\cK)$ is an admissible pair for appropriately chosen $\cK$.
\end{example}

\begin{remark}
    In order to maximise potential applications, the present article has been written with a high level of generality in mind. In particular, the coefficient ring $\cR$ is rather general and Euler systems are, for the most part, of arbitrary rank. To facilitate an initial read-through, the reader may wish to restrict their attention to the more classical rank-one situation. For example, one could take $K = \QQ, \cR = \ZZ_p, \cK$ to be the maximal totally real sub-extension of $\QQ^\mathrm{ab}$ and $T = \ZZ_p(1)$. In this case $r = 1$ and one has, for each $E \in \Omega(\cK/K)$, a canonical isomorphism
    \begin{align*}
        \bidual_{\cR[\cG_E]}^r H^1(\cO_{E,S(E)},T) \cong H^1(\cO_{E,S(E)},T)
    \end{align*}
    so that our definition of Euler system for the pair $(T,\cK)$ coincides with that of Rubin in \cite[Def. II.1.1]{rubin} (and see also \cite[Cor. B.3.5]{rubin}). In this setting the reader is therefore free to dispose of all exterior biduals.
\end{remark}

\subsubsection{Statement of the question and main results}

In this section we review the main results of this article. For ease of reading, we may defer the full statement of a given result to the main body of the article and prefer to give a less precise version here. For the convenience of the reader, each result will be listed with the numbering of its corresponding precise version in parentheses.

At the outset we recall that for every admissible pair $(T,\cK)$, Burns and Sano have constructed in \cite[\S 2.5]{bs} a module of so-called `basic' Euler systems $\ES^b(T,\cK)$ (and we warn the reader that their $T$ corresponds to our $T^*(1)$). We remark that their construction is in terms of an auxiliary set of places $\Sigma$ disjoint from $S$ but the assumed validity of ($\mathrm{H_3}$) implies that we can take $\Sigma = \varnothing$.

Following Coleman's work on circular distributions (see, for example, \cite{BuSe}) and the article \cite{bdss} of Burns, Sano and Seo and the presently named author, we are led to formulate the following question refining the one posed in \cite[Rem. 1.6]{bs}.

\begin{question}\label{main-question}
    Is every element of $\ES_r(T,\cK)$ contained in $\ES^b(T,\cK)$?
\end{question}

As general evidence towards an affirmative answer to this question we use the techniques of Burns and Sano in \cite{bs} and Burns, Sano and Sakamoto in \cite{bss2} to prove the following result which can be seen as an analogue of \cite[Thm. 2.7]{bdss}.

\begin{proposition}[\ref{evidence-theorem}]\label{evidence-theorem-pre}
    Let $E \in \Omega(\cK/K)$ be a field that satisfies `standard hypotheses' for the representation $T$. Then for every Euler system $\eta \in \ES_r(T,\cK)$ there exists a basic Euler system in $\ES^b(T,\cK)$ that agrees with $\eta$ on $E$.
\end{proposition}

Now, the proof of the equivariant Tamagawa Number Conjecture for abelian extensions of $\QQ$ of Burns and Greither in \cite{bg} suggests to us (\textit{c.f.} Remark \ref{eimc-remark}) that in order to obtain results towards our question about Euler systems themselves (and not just at a fixed level $E \in \Omega(\cK/K)$ as in Proposition \ref{evidence-theorem-pre}), one must employ Iwasawa Theory. To this end we shall give an Iwasawa-theoretic reinterpretation, modulo the validity of the weak Leopoldt conjecture, of the construction of $\ES^{b}(T,\cK)$.

This reinterpretation then allows us to establish the following criterion for an Euler system to be basic in which $H^2_\Iw(\mathcal{O}_{E,S(E)}, T)$ is an Iwasawa cohomology group.

\begin{theorem}[\ref{explicit-conditions-for-conj-theorem}]\label{basic-criterion-theorem}
    Let $\eta \in \ES^r(T,\cK)$ be an Euler system and assume that the $\mu$-invariants of a family of natural Iwasawa modules vanish. Then $\eta \in \ES^b(T,\cK)$ if and only if for every $E \in \Omega(\cK'/K)$ and certain characters $\psi \in \widehat{\cG_E}$ for which $e_\psi \cdot (\eta_{E_n})_n$ is non-trivial one has
        \begin{align*}
            \begin{array}{c}
                \displaystyle\mathrm{char}_{\bLambda_{E,\psi}}\left(H^2_\Iw(\cO_{E,S(E)}, T)^\psi\right) 
                \;\;\bigg|\;\;
                \mathrm{char}_{\bLambda_{E,\psi}}\left(\frac{\displaystyle\left(\varprojlim_{n \in \NN} \bidual_{\cR[\cG_{E_n}]}^r H^1(\cO_{E_n, S(E_n)}, T)\right)^\psi}{\displaystyle\braket{(\eta_{E_n})_n}^\psi_{\bLambda_{E,\psi}}}\right)
            \end{array}
        \end{align*}
        where $\bLambda_{E,\psi} := \cR(\im(\psi))[[\Gal(E_\infty/E)]]$. Moreover, $\eta$ is an $\cR[[\Gal(\cK/K)]]$-basis of $\ES^b(T,\cK)$ if and only if equality of these characteristic ideals holds.
\end{theorem}

In the setting of the multiplicative group over number fields, we then use this criterion to establish the following result concerning Question \ref{main-question}.

\begin{theorem}[\ref{question-valid-theorem}]\label{pre-question-valid-theorem}
    Let $\psi: G_\QQ \to \overline{\QQ_p}^\times$ be an abelian character of finite prime-to-$p$ order satisfying technical hypotheses. Let $T$ be the representation of $G_\QQ$ which is equal to $\cR := \ZZ_p(\im(\psi))$ as an $\cR$-module and upon which $G_\QQ$ acts via $\chi_{\mathrm{cyc}}\psi^{-1}$ where $\chi_\mathrm{cyc}$ is the cyclotomic character of $\QQ$. Then for an appropriately chosen $\cK$, Question \ref{main-question} has an affirmative answer for the pair $(T,\cK)$.
\end{theorem}

While this Theorem is stated for characters of $G_\QQ$ we are able to obtain in Theorem \ref{divisibilities-valid-higher-rank-theorem} a similar but somewhat weaker result for general totally real fields $K$.

Combining Theorem \ref{pre-question-valid-theorem} with the validity of the Iwasawa main conjecture for abelian fields, as proved by Greither in \cite[Thm. 3.1]{greither}, we then obtain the following Corollary.

\begin{corollary}[\ref{question-valid-corollary}]\label{pre-question-valid-corollary}
    With hypotheses as in Theorem \ref{question-valid-theorem}, $\ES_1(T, \cK)$ is a free\\ $\cR[[\Gal(\cK/K)]]$-module of rank one with a basis given by the Euler system of cyclotomic units.
\end{corollary}

Finally, we shall state a precise version of, and prove, the following result which demonstrates that something stronger than the result of Theorem \ref{evidence-theorem-pre} can be said about the question of to what extent arbitrary Euler systems over general representations are basic.

\begin{theorem}[\ref{kato-theorem}]\label{kato-theorem-pre}
    Let $p > 3$ and $C/\QQ$ be an elliptic curve satisfying technical hypotheses. Denote by $T_p(C)$ the $p$-adic Tate module of $C$. Then for any Euler system $\eta \in \ES_1(T_p(C),\cK)$ there exists a basic Euler system that agrees with $\eta$ on the cyclotomic $\ZZ_p$-extension $\QQ_\infty$ of $\QQ$. In particular, this applies to Kato's Euler system $z^{\mathrm{Kato}}$ (which is reviewed in, for example, \cite[\S 6.3]{bss2}).
\end{theorem}

\begin{acknowledgments}
    It is a pleasure for the presently named author to thank David Burns for suggesting this project to him as well as for numerous stimulating conversations and comments concerning the present manuscript. He also wishes to thank the anonymous referee for their comments and suggestions which have improved the exposition of the article.
    
    The author would additionally like to extend his gratitude to Dominik Bullach for his careful reading of earlier versions of this article, as well as Andrew Graham, Daniel Macias Castillo, and Takamichi Sano for their helpful comments and discussions.

    The present article is based on material contained in the author's King's College London Ph.D. thesis \cite{daoud}, the preparation and writing of which was financially supported by the Engineering and Physical Sciences Research Council [EP/N509498/1].
\end{acknowledgments}

\section{Evidence via the theory of higher rank Kolyvagin and Stark systems}\label{evidence-section}

In this section we provide unconditional evidence for Question \ref{main-question} via the theory of equivariant Kolyvagin and Stark systems. We remark the extension $\cK/K$ need not be assumed to satisfy the decomposition in Hypothesis \ref{k-hypothesis}(2) herein and so we can, and will, ignore this assumption during the course of this section.

\subsection{The set-up}\label{standard-hypotheses-section}

We fix a uniformiser $\varpi$ of $R$ and write $\kk$ for the residue field $R/\varpi$. We denote by $\overline{T}$ the residual representation $T \otimes_{R} \kk$.

Given a $G_K$-module $A$ we write $K(A)$ for the minimal Galois extension of $K$ such that $G_{K(A)}$ acts trivially on $A$. We moreover denote
\begin{align*}
    K_{p^n} := K(1)K(\mu_{p^n}, (\cO_K^\times)^{1/p^n}),\quad K_{p^\infty} := \bigcup_{n > 0} K_{p^n}
\end{align*}
where $K(1)$ is the Hilbert class field of $K$.

For any field $E \in \Omega(\cK/K)$ we then write
\begin{align*}
    E_{p^n} := EK_{p^n},\quad E_{p^\infty} := EK_{p^\infty},\quad E(T)_{p^n} := E_{p^n}K(T/p^nT),\quad E(T)_{p^\infty} := E_{p^\infty}K(T)
\end{align*}
and denote $\cT = \cT_E$ for the induced representation $\Ind_{G_K}^{G_E}(T)$. Given this, we then have the following `standard hypotheses' on $T$ and $E$:
\begin{enumerate}
    \item[($\mathrm{SH}_0$)]{For almost all primes $\mathfrak{q}$ of $K$, the map $\Fr_\mathfrak{q}^{p^k}-1$ is injective on $T$ for all $k \geq 0$.}
    \item[($\mathrm{SH}_1$)]{$\overline{T}$ is an irreducible $\kk[G_K]$-module.}
    \item[($\mathrm{SH}_2$)]{There exists $\tau \in G_{E_{p^\infty}}$ such that $T/(\tau-1)T$ is a free $\cR$-module of rank one.}
    \item[($\mathrm{SH}_3$)]{The cohomology groups $H^1(E(T)_{p^\infty}/K, \overline{T})$ and $H^1(E(T)_{p^\infty}/K, \overline{T}^\lor(1))$ vanish.}
    \item[($\mathrm{SH}_4$)]{If $p=3$ then $\overline{T}$ and $\overline{T}^\lor(1)$ have no non-zero isomorphic $R[G_K]$-subquotients.}
    \item[($\mathrm{SH}_5$)]{For all $\mathfrak{q}$ in $S_\mathrm{min}(\cT_E)\backslash S_\infty(K)$, the cohomology group $H^0(K_\mathfrak{q}, \overline{T}^\lor(1))$ vanishes.}
\end{enumerate}

\begin{example}\label{standard-hypotheses-example}There are several examples of representations $T$ and fields in $\Omega(\cK/K)$ for which all of the above hypotheses are satisfied:
\begin{enumerate}
    \item{Let $K$ be a totally real field and $\psi: G_K \to \overline{\mathbb{Q}_p}^{\times}$ an even abelian character $\psi \neq 1$ of finite order. Let $T$ be the representation of $G_K$ which is equal to $\cR := \ZZ_p(\im(\psi))$ as an $\cR$-module and upon which $G_K$ acts via $\chi_\mathrm{cyc}\psi^{-1}$. If $E \in \Omega(\cK/K)$ is a field with the property that no finite place in $S(E)$ splits completely in $L$ then it is shown in \cite[Lem. 5.3]{bss2} that $T$ and $E$ satisfy all the hypotheses $(\mathrm{SH}_0)$ through $(\mathrm{SH}_5)$.}
    \item{Let $C/\QQ$ be an elliptic curve and suppose that the image of the representation $\rho: G_\QQ \to \mathrm{GL}_2(\ZZ_p) \cong \Aut(T)$ contains $\mathrm{SL}_2(\ZZ_p)$ where $T$ is the $p$-adic Tate module of $C$. If $E \in \Omega(\cK/K)$ is a field with the property that for every finite place $v \in S(E)$ the curve $C(\QQ_v)$ has no points of order $p$ then it is shown in \cite[Lem. 6.17]{bss2} that $T$ and $E$ satisfy all the hypotheses $(\mathrm{SH}_0)$ through $(\mathrm{SH}_5)$.}
\end{enumerate}
\end{example}

\subsection{The statement and proof of the result}

\subsubsection{Vertical determinantal systems and basic Euler systems}
In order to provide the statement of the precise version of Proposition \ref{evidence-theorem-pre}, we first briefly recall the definition of basic Euler systems due to Burns and Sano in \cite{bs}.

Throughout the sequel we use the complex $C_{E,S(E)}(T)$ recalled in Appendix \ref{cohomology-appendix} and write $\Theta_{T,E}$ for the projection map
\begin{align*}
    \Theta_{T,E}: \Det_{\cR[\cG_E]}(C_{E,S(E)}(T)) \to \bidual_{\cR[\cG_E]}^r H^1(\cO_{E,S(E)}, T)
\end{align*}
constructed in \cite[\S2.6.3]{bs} (denoted by $\Pi_F$ in \textit{loc. cit.}).

We then write $\VS(T,\cK)$ for the $\cR[[\Gal(\cK/K)]]$-module of `vertical determinantal systems' for the pair $(T,\cK)$, whose definition we now recall from \cite[Def. 2.9]{bs}. 

We set
\begin{align*}
    \VS(T,\cK) = \varprojlim_{E \in \Omega(\cK/K)} \Det_{\cR[\cG_E]}(C_{E,S(E)}(T))
\end{align*}
where the transition map $\nu_{E'/E}$ is defined to be the composite
\begin{align*}
    \nu_{E'/E}: \Det_{\cR[\cG_{E'}]}(C_{E',S(E')}(T)) &\twoheadrightarrow \Det_{\cR[\cG_{E}]}(C_{E,S(E')}(T))\\
    &\xrightarrow{\sim} \Det_{\cR[\cG_{E}]}(C_{E,S(E)}(T)) \otimes \bigotimes_{v \in S(E')\setminus S(E)} \Det_{\cR[\cG_E]}(\R\Gamma_f(K_v, T_E))^\#\\
    &\xrightarrow{\sim} \Det_{\cR[\cG_{E}]}(C_{E,S(E)}(T))
\end{align*}
Here the first arrow is induced by taking $\Gal(E'/E)$-coinvariants, the second arrow induced by the exact triangle (\ref{change-U-triangle}) and the final arrow is induced by resolving each term $\R\Gamma_f(K_v, T_E)$ via the resolution (\ref{finite-support-resolution}), applying the trivialisation-by-the-identity map and finally the involution $\cR[\cG_E]^\# \xrightarrow{\sim} \cR[\cG_E]$.

By \cite[Thm. 2.18]{bs} one knows that the tuple of maps $(\Theta_{T,E})_E$ induces a homomorphism
\begin{align*}
    \Theta_{T,\cK}: \VS(T,\cK) \to \ES_r(T,\cK)
\end{align*}

\subsubsection{The statement of the result}

For the rest of this section we assume that the pair $(T,\cK)$ is chosen so as to satisfy $\mathrm{(H_0)}$, $\mathrm{(SH_0)}$, $\mathrm{(SH_1)}$, $(\mathrm{SH_4})$ and Hypothesis \ref{k-hypothesis}.

We remark that if this is the case then \cite[Lem. 3.11]{bss2} implies that $(T,\cK)$ satisfies the properties $\mathrm{(H_1)}$ through $\mathrm{(H_3)}$. As such, the pair $(T,\cK)$ is admissible.

Given $E \in \Omega(\cK/K)$ we say that $E$ `satisfies standard hypotheses for $T$' if $T$ and $E$ satisfy the additional hypotheses ($\mathrm{SH}_2$) and ($\mathrm{SH}_5$).

The following is the precise version of Proposition \ref{evidence-theorem-pre}.

\begin{proposition}\label{evidence-theorem}
    Let $\eta \in \ES_r(T,\cK)$ be an Euler system. Then for any $E \in \Omega(\cK/K)$ satisfying standard hypotheses we have
    \begin{align*}
        \eta_E \in \Theta_{T,E}(\Det_{\cR[\cG_E]}(C_{E,S(E)}(T)))
    \end{align*}
    In particular, there exists a basic Euler system $c \in \ES^b(T,\cK)$ with the property that $c_E = \eta_E$.
\end{proposition}

In order to prove this result we shall reduce it to a statement about equvivariant Kolyvagin and Stark systems. To do this, we first recall several important constructions related to these systems that are introduced in \cite{bs} and \cite{bss}.

\subsubsection{Kolyvagin and Stark systems}

We denote by $\cF$ the relaxed Selmer structure on $\cT_E$ as defined, for example, in \cite[Ex. 2.4]{bss2}. We remark that since ($\mathrm{SH}_5$) is satisfied in this case, $\cF$ coincides with both $\cF_{\mathrm{ur}}$ and $\cF_{\mathrm{can}}$ - the unramified and canonical Selmer structures respectively (\cite[Lem. 3.10]{bss2}).

We write $\KS_r(\cT, \cF)$ (resp. $\StS_r(\cT,\cF)$) for the $\cR[\cG_E]$-modules of Kolyvagin systems (resp. Stark systems) of rank $r$ for the pair $(\cT,\cF)$ as are defined in \cite[Def. 5.24]{bss} (resp. \cite[Def. 4.11]{bss}).

Then \cite[Thm. 3.6]{bss2} implies that, under the present hypotheses, there exists a canonical `Kolyvagin derivative' homomorphism
\begin{align*}
    \cD_\cT: \ES_r(T,\cK) \to \KS_r(\cT, \cF)
\end{align*}
and a canonical `regulator' isomorphism
\begin{align*}
    \cR_\cT: \StS_r(\cT,\cF) \xrightarrow{\sim} \KS_r(\cT,\cF)
\end{align*}

\subsubsection{Horizontal determinantal systems}
In order to ease notation in the sequel we shall henceforth write, for each $m \geq 1$, $\fr_m:= \cR/\varpi^m$ and $\cA_m := \cT/\varpi^m\cT$. In \cite[Def. 3.2]{bs}, the authors define a module of `horizontal determinantal systems' for the representation $\cA_m$. We briefly outline its definition here. To do this we first write, for each $m \geq 1$, $\cP_m$ for the set of primes $\mathfrak{q} \not\in S(\cF)$ of $K$ such that
\begin{itemize}
    \item{$\mathfrak{q}$ splits completely in $E_{p^m}$;}
    \item{$\cA_m/(\Fr_\mathfrak{q}-1)\cA_m \cong \mathfrak{r}_m$ as $\mathfrak{r}_m$-modules.}
\end{itemize}
We write $\cN_m = \cN_m(\cP_m)$ for the set of square-free products of primes in $\cP_m$ and denote, for each $\mathfrak{n} \in \cN_m$, $S_\mathfrak{n} := S \cup \set{ \mathfrak{q} \mid \mathfrak{n}}$ and $\nu(\fn)$ the number of prime divisors of $\fn$.

The $\mathfrak{r}_m[\cG_E]$-module of `horizontal determinantal systems' for the representation $\cA_m$ is given by
\begin{align*}
    \HS(\cA_m) := \varprojlim_{\mathfrak{n} \in \cN_m} \Det_{\mathfrak{r}_m[\cG_E]}(C_{K,S_\mathfrak{n}}(\cA_m))
\end{align*}
with respect to a particular bijective transition map where we write $C_{K,S_\mathfrak{n}}(\cA_m)$ for the complex $C_{K,S_\mathfrak{n}}(\cT_E) \otimes_R R/\varpi^mR$. The inclusion $\cN_{m+1} \subseteq \cN_m$ combines with the canonical codescent isomorphism
\begin{align*}
    \Det_{\mathfrak{r}_{m+1}[\cG_E]}(C_{K,S_\mathfrak{n}}(\cA_{m+1})) \otimes_{\fr_{m+1}[\cG_E]} \fr_{m}[\cG_E] \cong \Det_{\mathfrak{r}_{m}[\cG_E]}(C_{K,S_\mathfrak{n}}(\cA_{m}))
\end{align*}
for each $\fn \in \cN_{m+1}$ to imply the existence of a canonical surjective transition map
\begin{align*}
    \kappa_{m+1}: \HS(\cA_{m+1}) \to \HS(\cA_m)
\end{align*}
We now define the $\cR[\cG_E]$-module of horizontal determinantal systems for the representation $\cT$ to be
\begin{align*}
    \HS(\cT) := \varprojlim_{m \in \NN} \HS(\cA_m)
\end{align*}
where the inverse limit is taken with respect to the $\kappa_m$.

One now checks that there is a commutative diagram
\begin{center}
    \begin{tikzcd}
        \HS(\cA_{m+1}) \arrow[r] \arrow[d, "\kappa_{m+1}"] &\StS_r(\cA_{m+1}, \cF)\arrow[d]\\
        \HS(\cA_{m}) \arrow[r] &\StS_r(\cA_m, \cF)
    \end{tikzcd}
\end{center}
where the horizontal maps are those given by \cite[Thm. 3.3]{bs} and the right-hand vertical map is the one defined in \cite[\S 4.3]{bss}. By passing to the limit over $m \in \NN$ one then deduces the existence of a canonical projection map
\begin{align*}
    \Psi_{\cT} : \HS(\cT) \to \StS_r(\cT,\cF)
\end{align*}
In fact, the assumed validity of ($\mathrm{SH}_5$) combines with \cite[Thm 3.3, Rem. 3.10]{bs} to imply that $\Psi_{\cT}$ is an isomorphism.

On the other hand, for each $m \geq 1$ there is a natural surjective projection map
\begin{align*}
    \VS(T,\cK) \twoheadrightarrow \Det_{\fr_m[\cG_E]}(C_{K,S}(\cA_m)) \xrightarrow{\sim} \HS(\cA_m)
\end{align*}
This projection map is clearly compatible with the $\kappa_m$ so that the universal property of the inverse limit implies the existence of a surjection
\begin{align*}
    \Pi_{\cT}: \VS(T,\cK) \twoheadrightarrow \HS(\cT)
\end{align*}

\subsubsection{Proof of Proposition \ref{evidence-theorem}}

Given the constructions of the previous section, we are now in a position to prove Proposition \ref{evidence-theorem}. We first observe that by taking the inverse limit over $m$ in the diagram of \cite[Thm. 4.16]{bs} and applying \cite[Thm. 3.6(i)]{bss2} one deduces the existence of a commutative diagram
\begin{equation}\label{reduction-diagram}
    \begin{tikzcd}[row sep=3em, column sep=3em]
        \VS(T,\cK) \arrow[r, "\Theta_{T,\cK}"] \arrow[d, twoheadrightarrow, "\Psi_\cT \circ \Pi_{\cT}"]&\ES_r(T,\cK) \arrow[d, "\cD_\cT"]\\
        \StS_r(\cT,\cF) \arrow[r, "\cR_\cT", "\sim"'] &\KS_r(\cT,\cF)
    \end{tikzcd}
\end{equation}
of $\cR[[\Gal(\cK/K)]]$-modules. By the commutativity of this diagram, one knows that the composite $\cD_\cT \circ \Theta_{T,\cK}$ is surjective. As such, if we are given an $\cR[[\Gal(\cK/K)]]$-basis $z$ of $\VS(T,\cK)$ then $(\cD_\cT \circ \Theta_{T,\cK})(z)$ is an $\cR[\cG_E]$-generator of $\KS_r(\cT,\cF)$. Now fix an Euler system $\eta \in \ES_r(\cT,\cK)$. Then there exists $x \in \cR[\cG_E]$ such that $\cD_{\cT}(\eta) = x\cdot (\cD_\cT \circ \Theta_{T,\cK})(z)$. Evaluating this equality of equivariant Kolyvagin systems at the empty product ideal $1$, we have that
\begin{align*}
    \eta_E = \cD_{\cT}(\eta)_1 = x\cdot (\cD_\cT \circ \Theta_{T,\cK})(z)_1 = x\cdot \Theta_{T,\cK}(z)_E \in \Theta_{T,E}(\Det_{\cR[\cG_E]}(C_{E,S(E)}(T)))
\end{align*}
This establishes the first assertion of Proposition \ref{evidence-theorem}.

Finally, in order to construct a basic Euler system with the stated property, it suffices to choose any pre-image $\overline{x}$ of $x$ under the natural surjective map $\cR[[\Gal(\cK/K)]] \to \cR[\cG_E]$ and set $c := \Theta_{T,\cK}(\overline{x} \cdot z)$. Indeed, it is then clear that $c \in \ES^b(T,\cK)$ and that
\begin{align*}
    c_E = \Theta_{T,\cK}(\overline{x}\cdot z)_E = x\cdot \Theta_{T,\cK}(z)_E = \eta_E
\end{align*}
as required.
\qed

\section{Iwasawa theory of Euler systems}\label{iwasawa-theory-section}

In this section we shall give a useful Iwasawa-theoretic reinterpretation of the module of vertical systems $\VS(T,\cK)$ constructed by Burns and Sano in \cite[Def. 2.8]{bs}.

We recall that, by hypothesis, there is a decomposition $\cK = \cK'K_\infty$. For each $E \in \Omega(\cK'/K)$, we denote by $E_\infty$ the compositum $EK_\infty$, by $E_n$ the $n^{th}$ layer of the $\ZZ_p$-extension $E_\infty/E$ and by $\bLambda_{E}$ the equivariant Iwasawa algebra $\cR[[\Gal(E_\infty/K)]] \cong \cR[\cG_E][[\Gal(K_\infty/K)]]$. We moreover write $Q(\bLambda_E)$ for the total quotient ring of $\bLambda_E$. Similarly, we write $\Lambda = \cR[[\Gal(E_\infty/E)]] = \cR[[\Gal(K_\infty/K)]]$ for the non-equivariant Iwasawa algebra (which is independent of $E$) and $Q(\Lambda)$ for its total ring of quotients. Finally, we shall denote for each $i \in \set{1,2}$ and $n \in \NN$ the functor
\begin{align*}
    H^i_{\Iw}(\cO_{E,S(E)},-) := \varprojlim_{n \in \NN} H^i(\cO_{E_n,S(E_n)}, -)
\end{align*}
computing Iwasawa cohomology

\subsection{Iwasawa theoretic setup}

Throughout the rest of this section we assume the following conjecture:

\begin{conjecture}[Weak Leopoldt Conjecture]\label{weak-leopoldt-conjecture}
    For every $E \in \Omega(\cK'/K)$, $H^2_\Iw(\cO_{E,S(E)}, T)$ is a torsion $\Lambda$-module.
\end{conjecture}

\begin{remark}\label{weak-leopolodt-remark}
    The above formulation of Conjecture \ref{weak-leopoldt-conjecture} is due to Perrin-Riou in \cite[\S 1.3]{perrin-riou1}. It is known to be true in several settings in arithmetic of which we mention the following two cases:
    \begin{enumerate}
        \item{When $T = \ZZ_p(1)$ and $K_\infty/K$ is the cyclotomic $\ZZ_p$-extension then the weak Leopoldt conjecture holds due to a result of Iwaswawa \cite{iwasawa} (and see also \cite[\S 1.3 Rem (ii)]{perrin-riou1}).}
        \item{When $K = \QQ$ and $T$ is the $p$-adic Tate module of an elliptic curve over $K$ then the validity of the weak Leopoldt conjecture follows from a result of Kato in \cite[Thm. 12.4(i)]{kato3}}.
    \end{enumerate}
\end{remark}

Fix now a field $E \in \Omega(\cK'/K)$ and a finite set of places $U$ of $K$ containing $S$. We define a complex by means of the derived limit
\begin{align*}
    C_{E_\infty, U}(T) := \varprojRlim_{n \in \NN} C_{E_n,U(E_n)}(T)
\end{align*}
of $\bLambda_E$-modules. Then by \cite[Prop. 3.5]{bd} one knows that $C_{E_\infty, U}(T)$ is perfect and there is a canonical identification
\begin{align*}
    H^0(C_{E_\infty, U}(T)) \cong \varprojlim_{n \in \NN} H^1(\cO_{E_n,U(E_n)}, T)
\end{align*}
and an exact sequence
\begin{align}
    0 \to \varprojlim_{n \in \NN} H^2(\cO_{E_n,U}, T) \to H^1(C_{E_\infty, U}(T)) \to \varprojlim_{n \in \NN} Y_{E_n}(T)^* \to 0\label{iwasawa-cohomology-exact-sequence}
\end{align}
of $\bLambda_E$-modules.

By the assumed validity of the weak Leopoldt conjecture, we have a natural projection map $\Theta^\infty_{T,E}$ given by the composite
\begin{align*}
    &\Det_{\bLambda_E}(C_{E_\infty, S(E)}(T))\\
    &\hookrightarrow \Det_{Q(\bLambda_E)}(Q(\bLambda_E) \otimes_{\bLambda_E} C_{E_\infty, S(E)}(T))\\
    &\xrightarrow{\sim} \left(Q(\bLambda_E) \otimes_{\bLambda_E} \exprod_{\bLambda_E}^r \varprojlim_{n \in \NN} H^1(\cO_{E_n, S(E_n)}, T)\right) \otimes_{Q(\bLambda_E)} \left(Q(\bLambda_E) \otimes_{\bLambda_E} \exprod_{\bLambda_E}^r \varprojlim_{n \in \NN} Y_{E_n}(T)^*\right)^*\\
    &\xrightarrow{\sim} Q(\bLambda_E) \otimes_{\bLambda_E} \exprod_{\bLambda_E}^r \varprojlim_{n \in \NN} H^1(\cO_{E_n, S(E_n)}, T)
\end{align*}
where the second arrow is the canonical passage-to-cohomology map and the third arrow is induced by fixing a family of compatible bases of the free rank $r$ modules $Y_{E_n}(T)^*$ for all $n \in \NN$.

\subsubsection{Projection maps and biduals}

In this section we record the following useful Lemma in which we use the theory of standard representatives introduced in \cite[\S A.4]{bs}.

\begin{lemma}\label{bidual-compatibility-lemma}\text{}
    \begin{enumerate}
        \item{There is a canonical identification of $\bLambda_E$-modules
            \begin{align}
                \bidual_{\bLambda_E}^r H^1_\Iw(\cO_{E,S(E)}, T) \xrightarrow{\sim} \varprojlim_{n \in \NN} \bidual_{\cR[\cG_{E_n}]}^r H^1(\cO_{E_n,S(E_n)}, T)\label{bidual-inverse-limit-iso}
            \end{align}} 
        \item{There exists a quadratic standard representative $(P, \psi, \set{b_1,\dots,b_d})$ of the complex $C_{E_\infty, S(E)}(T)$ with respect to the surjection
            \begin{align*}
                H^1(C_{E_\infty, S(E)}(T)) \to \varprojlim_{n \in \NN} Y_{E_n}(T)^*
            \end{align*}
        }
        \item{The image of $\Theta_{T,E}^\infty$ is contained in $\bidual_{\bLambda_E}^r H^1_\Iw(\cO_{E,S(E)}, T)$}
    \end{enumerate}
\end{lemma}

\begin{proof}

Part (1) is given by applying \cite[Lem. B.15]{sakamoto} to the complex $C_{E_\infty, S(E)}(T)$ and Part (2) is proved in \cite[Lem. 7.10]{BuSaNC}.

We prepare for the proof of the remainder of the above Lemma by first collecting some useful notations. We write $\mathfrak{S}_n$ for the permutation group on $m$ elements and for any $1 \leq n \leq m$ we write
\begin{align*}
    \mathfrak{S}_{m,n} := \set{\sigma \in \mathfrak{S}_s | \sigma(1) < \cdots < \sigma(n) \text{ and } \sigma(n+1) \leq \cdots \leq \sigma(m)}
\end{align*}
Given a commutative ring $R$ and a finitely generated $R$-module $X$, suppose to be given $\phi \in X^*$. For every $m \geq 1$ we define a map
\begin{align*}
    \exprod_R^m X &\to \exprod_R^{m-1} X\\
    x_1 \wedge \cdots \wedge x_m &\mapsto \sum_{1 \leq i \leq m} (-1)^{i-1}(x_1 \wedge \cdots \wedge x_{i-1} \wedge x_{i+1} \wedge \cdots \wedge x_m)
\end{align*}
which, by abuse of notation, we also denote by $\phi$. Now given a family of maps $\set{\phi_i}_{1 \leq i \leq m}$ in $X^*$, we define for every pair of natural numbers $m > n$ the map $\exprod_{n < i \leq m} \phi_i$ to be the composite $\phi_{n+1} \circ \cdots \circ \phi_m$.

Given these notations we now finish the proof of the Lemma. For each $1 \leq i \leq d$ denote $\psi_i := b_i^*\circ \psi$ where $(b_i^*)_i$ is the dual basis to $(b_i)_i$ in $P^*$.

By \cite[Lem. A.7(i),(ii)]{bs} one knows that $\Theta_{T,E}^\infty$ coincides with the map induced by the assignment
\begin{align*}
    \Det_{\bLambda_E}(C_{E_\infty, S(E)}(T)) = \exprod_{\bLambda_E}^d P \otimes \exprod_{\bLambda_E}^d P^* &\to \exprod_{\bLambda_E}^r P\\
    (b_1 \wedge \cdots \wedge b_d) \otimes (b_1^* \wedge \cdots \wedge b_d^*) &\mapsto \left(\exprod_{r < i \leq d} \psi_i\right)(b_1 \wedge \cdots \wedge b_d)
\end{align*}

On the other hand, if one denotes for each $1 \leq i \leq d$ and $n \in \NN$ the image of $b_{i}$ under the codescent map $P \to P_n$ by $b_{n,i}$ then $\set{b_{n,i}}_{1 \leq i \leq d}$ is a basis of $P_n$ and the complex $[P_{n} \xrightarrow{\psi_{n}} P_{n}]$ constitutes a quadratic standard representative of $C_{E_n,S(E_n)}(T)$.

In particular, if we write $\psi_{n,i} = b_{n,i}^* \circ \psi_n$ then the map induced by the assignment
\begin{align}\label{finite-projection-map}
    \Det_{\cR[\cG_{E_n}]}(C_{E_n, S(E_n)}(T)) = \exprod_{\cR[\cG_{E_n}]}^d P_n \otimes \exprod_{\cR[\cG_{E_n}]}^d P_n^* &\to \exprod_{\cR[\cG_{E_n}]}^r P_n\notag\\
    (b_{n,1} \wedge \cdots \wedge b_{n,d}) \otimes (b_{n,1}^* \wedge \cdots \wedge b_{n,d}^*) &\mapsto \left(\exprod_{r < i \leq d} \psi_{n,i}\right)(b_{n,1} \wedge \cdots \wedge b_{n,d})
\end{align}
has image inside $\bidual_{\cR[\cG_{E_n}]}^r H^1(\cO_{E_n,S(E_n)}, T)$ by \cite[Lem. A.11(ii)]{bs}. The explicit description of these projection maps then implies that we have a commutative diagram
\begin{equation}\label{determinant-comparison-diagram}
    \begin{tikzcd}[row sep=2.5em, column sep=5em]
        \displaystyle\Det_{\bLambda_E}(C_{E_\infty, S(E)}(T)) \arrow[d, "\rho_{T,E}"', "\sim" slanted] \arrow[r, "\Theta_{T,E}^\infty"] &\displaystyle\exprod_{\bLambda_E}^r P \arrow[d, "\sim" slanted]\\
        \displaystyle\varprojlim_n \Det_{\cR[\cG_{E_n}]}(C_{E_n, S(E_n)}(T)) \arrow[r, "\varprojlim_n \Theta_{T,E_n}"] &\displaystyle\varprojlim_{n \in \NN} \exprod_{\cR[\cG_{E_n}]}^r P_n
    \end{tikzcd}
\end{equation}
where the bottom left limit is taken with respect to the maps $\nu_{E_m/E_n}$, the left hand vertical arrow is induced by the codescent isomorphism (which is an isomorphism since it is a surjective map of free $\bLambda_E$-modules of rank one), the right hand vertical arrow is the natural one and the bottom arrow is the limit over $n$ of the map (\ref{finite-projection-map}).

The Lemma now follows by noting that the isomorphism given in the first part of the statement is induced by the right hand vertical map in the above diagram.
\end{proof}

\subsection{Reinterpreting vertical determinantal systems}

To ease notation we denote, for any $E \in \Omega(\cK'/K)$ and finite place $v$ of $K$ not in $S(E)$, the complex of $\bLambda_E$-modules given by the derived limit
\begin{align*}
    P_{E_\infty, v}(T) := \varprojRlim_{n \in \NN} \R\Gamma_f(K_v, T_{E_n}^*(1))
\end{align*}
Observe that $P_{E_\infty,v}(T)$ is isomorphic in $\Der^p(\bLambda_E)$ to the complex
\begin{align}
    \left[{\varprojlim_{n \in \NN} T^*_{E_n(1)}} \xrightarrow{(1-\Fr_{v}^{-1})_{n \in \NN}} {\varprojlim_{n \in \NN} T^*_{E_n}(1)}\right]\label{P-E-infinity-resolution}
\end{align}
where the first term is placed in degree 0. Now fix a pair of fields $E \subseteq E'$ in $\Omega(\cK'/K)$. Then by passing to the limit over $n$ in the exact triangle (\ref{change-U-triangle}) one deduces the existence of an exact triangle
\begin{align}
    C_{E_\infty, S(E')}(T) \to C_{E_\infty, S(E)}(T) \to \bigoplus_{v \in S(E')\setminus S(E)} P_{E_\infty, v}(T)^*[-1]\label{change-U-triangle-infinity}
\end{align}
in $\Der^p(\bLambda_E)$.
We now define a map
\begin{align}
    \nu_{E'/E}^\infty: \Det_{\bLambda_{E'}}(C_{E'_\infty, S(E')}(T)) &\twoheadrightarrow \Det_{\bLambda_E}(C_{E_\infty, S(E')}(T))\tag{$\nu_{E'/E}^{\infty,1}$}\\
    &\xrightarrow{\sim} \Det_{\bLambda_E}(C_{E_\infty, S(E)}(T)) \otimes \bigotimes_{v \in S(E')\setminus S(E)} \Det_{\bLambda_E}(P_{E_\infty, v}(T))^\#\notag\\
    &\xrightarrow{\sim} \Det_{\bLambda_E}(C_{E_\infty,S(E)}(T))\tag{$\nu_{E'/E}^{\infty,2}$}
\end{align}
where the first arrow is induced by base-change to $\bLambda_{E}$ and the inverse limit of the codescent isomorphisms (\ref{codescent-iso}), the second arrow is induced by the exact triangle (\ref{change-U-triangle-infinity}), and the third arrow is induced by resolving each term $P_{E_\infty, v}$ as in (\ref{P-E-infinity-resolution}), applying the canonical trivialisation-by-the-identity map
\begin{align*}
    \Det_{\bLambda_E}(P_{E_\infty, v}(T)) = \Det_{\bLambda_E}\left(\varprojlim_{n \in \NN} T^*_{E_n}(1)\right) \otimes_{\bLambda_E} \Det_{\bLambda_E}^{-1}\left(\varprojlim_{n \in \NN} T^*_{E_n}(1)\right) \xrightarrow{\sim} \bLambda_E
\end{align*}
and then the involution isomorphism $\bLambda_E^\# \xrightarrow{\sim} \bLambda_E$.

Henceforth we shall use the natural identification of rings $\varprojlim_{E \in \Omega(\cK/K)} \bLambda_E \cong \cR[[\Gal(\cK/K)]]$ without further explicit reference.

In the following Lemma we write $\VS_\Iw(T,\cK)$ for the $\cR[[\Gal(\cK/K)]]$-module given by the inverse limit $\varprojlim_{E \in \Omega(\cK'/K)} \Det_{\bLambda_E}(C_{E_\infty, S(E)}(T))$ taken with respect to the maps $\nu_{E'/E}^\infty$.

\begin{lemma}\label{vertical-systems-reinterpretation-proposition}
    There is a commutative diagram of $\cR[[\Gal(\cK/K)]]$-modules
    \begin{equation}\label{vs-infinity-diagram-equation}
        \begin{tikzcd}[row sep = 4em, column sep=4em]
            \VS_\Iw(T,\cK) \arrow[dr, "\Theta_{T,\cK}^\infty", hookrightarrow] \arrow[d, "\rho_{T,\cK}"', "\sim"' slanted]\\
            \VS(T,\cK) \arrow[r, "\Theta_{T,\cK}", hookrightarrow] &\ES_r(T,\cK)
        \end{tikzcd}
    \end{equation}
    where $\rho_{T,\cK}$ is the tuple of maps $(\rho_{T,E})_E$ and $\Theta_{T,\cK}^\infty$ is the composite
    \begin{align*}
        \VS_\Iw(T,\cK) &\xrightarrow{(\Theta_{T,E}^\infty)_E} \prod_{E \in \Omega(\cK'/K)} \bidual_{\bLambda_E}^r H^1_\Iw(\cO_{E,S(E)}, T)\\
        &\xrightarrow[\sim]{(\ref{bidual-inverse-limit-iso})} \prod_{E \in \Omega(\cK'/K)} \varprojlim_{n \in \NN} \bidual_{\cR[\cG_{E_n}]}^r H^1(\cO_{E_n,S(E_n)}, T)\\
        &\to \prod_{E \in \Omega(\cK/K)} \bidual_{\cR[\cG_E]}^r H^1(\cO_{E, S(E)}, T)
    \end{align*}
\end{lemma}

\begin{proof}
    We first observe that the diagram (\ref{determinant-comparison-diagram}) implies that $\Theta_{T,\cK}^\infty$ indeed has image inside $\ES_r(T,\cK)$. The Proposition will therefore follow immediately by passing to the inverse limit over $E \in \Omega(\cK'/K)$ of the diagrams (\ref{determinant-comparison-diagram}) once it is demonstrated that the map $\rho_{T,\cK}$ is well-defined.
    
    To do this we fix a pair of fields $E \subseteq E'$ in $\Omega(\cK'/K)$. Since for each $n \in \NN$, $\cR[\cG_{E_n}]$ is a reduced ring the determinant functor over $\cR[\cG_{E_n}]$ is functorial in exact triangles. A straightforward diagram chase using this fact to compare the maps $\nu_{E'/E}^\infty$ and $\nu_{E'/E}$ then gives the commutativity of the diagram
    \begin{equation}
        \begin{tikzcd}[row sep=3em, column sep=3em]\label{determinant-comparison-lemma-3}
            \Det_{\bLambda_{E'}}(C_{E'_\infty, S(E')}(T)) \arrow[r, "\rho_{T,E'}"] \arrow[d, "\nu_{E'/E}^\infty"] &\displaystyle\varprojlim_{n \in \NN} \Det_{\cR[\cG_{E_n}]}(C_{E_n',S(E_n')}(T)) \arrow[d, "\varprojlim_{n}\nu_{E_n'/E_n}"]\\
            \Det_{\bLambda_E}(C_{E_\infty, S(E)}(T)) \arrow[r, "\rho_{T,E}"] &\displaystyle \varprojlim_{n \in \NN} \Det_{\cR[\cG_{E_n}]}(C_{E_n,S(E_n)}(T))
        \end{tikzcd}
    \end{equation}
    where $\nu_{E'/E}$ is the transition map used in the definition of $\VS(T,\cK)$. By passing to the limit over $E \in \Omega(\cK'/K)$ of these diagarams it then follows that $\rho_{T,\cK}$ maps $\VS_\Iw(T,\cK)$ isomorphically onto $\VS(T,\cK)$ as claimed.
\end{proof}

\begin{remark}
    To ease exposition in the sequel we shall continue to use the notation $\VS_\Iw(T,\cK)$ to indicate the (same) construction $\VS(T,\cK)$ viewed with $\Omega(\cK'/K)$ as its indexing set via the isomorphism $\rho_{T,\cK}$.
\end{remark}

\subsection{Reinterpreting the basic condition}

In this subsection we show that the question of whether an Euler system is basic admits an explicit reformulation in terms of its values. To this end we first fix an Euler system $\eta \in \ES_r(T,\cK)$ and a field $E \in \Omega(\cK'/K)$. By the assumed validity of the weak Leopoldt conjecture, the map $Q(\bLambda_E) \otimes_{\bLambda_E} \Theta_{T,E}^\infty$ is bijective. We may thus define
\begin{align*}
    \cL_{\eta,E} \in \Det_{Q(\bLambda_E)}(Q(\bLambda_E) \otimes_{\bLambda_E} C_{E_\infty, S(E)}(T))
\end{align*}
to be the inverse image under $Q(\bLambda_E) \otimes_{\bLambda_E} \Theta_{T,E}^\infty$ of the element
\begin{align*}
    (\eta_{E_n})_{n \in \NN} \in \varprojlim_{n} \bidual_{\cR[\cG_{E_n}]}^r H^1(\cO_{E,S(E)}, T)
\end{align*}

\begin{proposition}\label{reinterpret-iwasawa-conjecture-proposition}
    Let $\eta \in \ES_r(T,\cK)$ be an Euler system. Then $\eta$ is basic if and only if for every field $E \in \Omega(\cK'/K)$ one has
    \begin{align}
        \cL_{\eta,E} \in \Det_{\bLambda_E}(C_{E_\infty, S(E)}(T)) \hookrightarrow \Det_{Q(\bLambda_E)}(Q(\bLambda_E) \otimes_{\bLambda_E} C_{E_\infty, S(E)}(T))\label{equivariant-divisibility-condition}
    \end{align}
    where the above embedding is the natural one induced by base-changing to $Q(\bLambda_E)$.
\end{proposition}

\begin{proof}
    Fix $\eta \in \ES_r(T,\cK)$ and first suppose that $\eta \in \ES^b(T,\cK)$. Then there exists a vertical system $z = (z_{E_\infty})_E \in \VS_\Iw(T,\cK)$ that projects under $\Theta_{T,\cK}^\infty$ to give $\eta$. But for each $E \in \Omega(\cK'/K)$, $\cL_{\eta,E}$ is clearly the unique element that projects under $Q(\bLambda_E) \otimes_{\bLambda_E}\Theta_{T,E}^\infty$ to give $(\eta_{E_n})_E$. Hence for each $E \in \Omega(\cK'/K)$, the element $z_{E_\infty}$ must coincide with $\cL_{\eta,E}$.
    
    Conversely, we claim that the family $\cL_\eta := (\cL_{\eta,E})_E$ constitutes an element of $\VS_\Iw(T,\cK)$ and thus projects under $\Theta_{T,\cK}^\infty$ to give $\eta$. In other words, we claim that for every pair of fields $E \subseteq E'$ one has that $\nu^\infty_{E'/E}(\cL_{\eta,E'}) = \cL_{\eta,E}$. By the definition of the map $\nu_{E'/E}^\infty$ we are reduced to verifying the pair of equalities $\nu_{E'/E}^{\infty,1}(\cL_{\eta,E'}) = \cP_{E'/E}\cdot\cL_{\eta,E}$ and $\nu_{E'/E}^{\infty,2}(\cP_{E'/E}\cdot\cL_{\eta,E}) = \cL_{\eta,E}$.
    
    To prove the first equality we observe that we have a commutative diagram:
    \begin{center}
        \begin{tikzcd}[row sep = 3em, column sep = 3em]
            \Det_{\bLambda_{E'}}(C_{E'_\infty, S(E')}(T)) \arrow[r, "\Theta_{T,E'}^\infty"] \arrow[d, "\nu_{E'/E}^{\infty,1}"] &\displaystyle\varprojlim_{n \in \NN} \bidual_{\cR[\cG_{E'_n}]}^r H^1(\cO_{E',S(E')}, T) \arrow[d, "N_{E'/E}^\infty"]\\
            \displaystyle \Det_{\bLambda_{E}}(C_{E_\infty, S(E')}(T)) \arrow[r, "\Theta_{T,E}^\infty"] &\displaystyle \varprojlim_{n \in \NN} \bidual_{\cR[\cG_{E_n}]}^r H^1(\cO_{E,S(E')}, T)
        \end{tikzcd}
    \end{center}    
    where $N_{E'/E}^\infty := \varprojlim_{n \in \NN} N_{E'_n/E_n}$. As such we may calculate
    \begin{align*}
        (\Theta_{T,E}^\infty \circ \nu_{E'/E}^{\infty,1})(\cL_{\eta, E'}) &= (N_{E'/E}^\infty \circ \Theta_{T,E'}^\infty)(\cL_{\eta,E'})\\
        &= N_{E'/E}^\infty((\eta_{E'_n})_n)\\
        &= \cP_{E'/E} \cdot (\eta_{E_n})_n\\
        &= \Theta_{T,E}^\infty(\cP_{E'/E} \cdot \cL_{\eta,E})
    \end{align*}
    whence the injectivity of $\Theta_{T,E}^\infty$ yields the desired equality.
    
    As for the second equality, we first observe that for each $v \in S(E')\setminus S(E)$ the complex
    \begin{align*}
        Q(\bLambda_E) \otimes_{\bLambda_E} \bigoplus_{v \in S(E')\setminus S(E)} P_{E_\infty, v}(T)
    \end{align*}
    is acyclic since, for any place $w$ of $E$ lying over $v$, $1-\Fr_w^{-1}$ is a non-zero-divisor in the algebra $\Lambda$.
    
    One then deduces the existence (see, for example, \cite[Lem. 1]{bf}) of a commutative diagram
    \begin{center}
        \begin{tikzcd}
            \displaystyle\Det_{Q(\bLambda_E)}\left(Q(\bLambda_E) \otimes_{\bLambda_E} \bigoplus_{v \in S(E')\backslash S(E)} P_{E_\infty, v}(T)\right) \arrow[r] \arrow[d] &Q(\bLambda_E) \arrow[d, "\cdot\mathcal{P}_{E'/E}^{-1}"]\\
            \Det_{Q(\bLambda_E)}(0) \otimes_{Q(\bLambda_E)} \Det_{Q(\bLambda_E)}^{-1}(0) \arrow[r] &Q(\bLambda_E)
        \end{tikzcd}
    \end{center}
    where $\cP_{E'/E} := \prod_{v \in S(E')\backslash S(E)} P_v(\Fr_v^{-1})$, the top and bottom maps are the relevant evaluation maps and the left-hand vertical map is the canonical passage-to-cohomology map. This diagram then gives rise to the further diagram (in which we abbreviate the functor $\otimes_\Lambda$ to just $\cdot$ and slightly abuse notation by using $\Theta_{T,E}^\infty$ to represent the projection map for both sets $S(E)$ and $S(E')$):
    \begin{center}
        \begin{tikzcd}[row sep = 3em, column sep = 3em]
            \displaystyle Q(\Lambda)\cdot\Det_{\bLambda_{E}}(C_{E_\infty, S(E')}(T)) \arrow[r, "\Theta_{T,E}^\infty"] \arrow[d, "\nu_{E'/E}^{\infty,2}"] &\displaystyle Q(\Lambda) \cdot \varprojlim_{n \in \NN} \bidual_{\cR[\cG_{E_n}]}^r H^1(\cO_{E,S(E')}, T) \arrow[d, "\cP_{E'/E}^{-1}"]\\
            \displaystyle Q(\Lambda)\cdot\Det_{\bLambda_{E}}(C_{E_\infty, S(E)}(T)) \arrow[r, "\Theta_{T,E}^\infty"] &\displaystyle Q(\Lambda) \cdot \varprojlim_{n \in \NN} \bidual_{\cR[\cG_{E_n}]}^r H^1(\cO_{E,S(E)}, T)
        \end{tikzcd}
    \end{center}    
    A similar calculation as before using this second diagram then shows that $\nu_{E'/E}^{\infty,2}(\cP_{E'/E}\cdot\cL_{\eta,E}) = \cL_{\eta,E}$ as claimed.
\end{proof}

\begin{remark}\label{eimc-remark}
    Proposition \ref{reinterpret-iwasawa-conjecture-proposition} is reminiscent of the equivariant Iwasawa Main Conjectures (eIMC) that is formulated by Burns, Kurihara and Sano in \cite[Conj. 3.1]{bks2} (and see also \cite[Rem. 3.3]{bks2}) which originate with the work of Kato in \cite{kato1} and \cite[\S 3.3.8]{kato2}. Indeed, if we specialise to the case $K = \mathbb{Q}, T = \ZZ_p(1)$ and $\eta$ the cyclotomic Euler system then the eIMC is the statement that, for every $E \in \Omega(\cK'/K)$, not only $\cL_{\eta,E} \in \Det_{\bLambda_E}(C_{E_\infty, S(E)}(T))$ but that it constitutes a $\bLambda_E$-basis for this module.
    
    In the next section we push this parallel further to see that the condition in the Proposition is effectively an equivariant version of the Iwasawa-theoretic divisibilities that Euler systems are expected to satisfy. As such we believe Question \ref{main-question} to be at the heart of the theory of Euler systems.
\end{remark}

\subsection{Reformulating the condition of Proposition \ref{reinterpret-iwasawa-conjecture-proposition}}

We now aim to reduce the condition of Proposition \ref{reinterpret-iwasawa-conjecture-proposition} to an explicit statement about non-equivariant Iwasawa-theoretic divisibilities and the vanishing of appropriate $\mu$-invariants.

\subsubsection{Equivariant Iwasawa Algebras}
We first review some important algebraic definitions and facts concerning equivariant Iwasawa algebras. The proofs of these results are essentially well-known, especially in the case that $\cR$ is a group ring, but for lack of a suitable reference for general orders $\cR$ we provide full proofs here.

Suppose to be given a field $E \in \Omega(\cK'/K)$. The ring $\cR[\cG_E]$ is profinite as it is a finitely generated $R$-algebra. In particular, \cite[Ex. 5.1.3(2)]{ribes-zalesskii} implies that $\cR[\cG_E]$ decomposes as a product $\cR[\cG_E] = \prod_{\chi \in I_E} \cR_\chi$ where each $\cR_\chi$ is a local ring. Note that, since a local ring has connected spectrum, the set $\set{e_\chi}_{\chi \in I_E}$ of idempotents of $\cR[\cG_E]$ inducing this decomposition constitutes a complete set of primitive orthogonal idempotents.

Similarly, we write $\set{e_\psi}_{\psi \in \overline{I_E}}$ for a complete set of primitive orthogonal idempotents of the semi-simple algebra $\cB[\cG_E]$ so that $\cB[\cG_E] = \prod_{\psi \in \overline{I_E}} \cB_\psi$ where each $\cB_\psi$ is a finite extension of $Q$.

We observe that if we are given $\chi \in I_E$, then there exists a decomposition $e_\chi = \sum_{i=1}^{d_\chi} e_{\psi(\chi)_i} \in \cB[\cG_E]$ for some $d_\chi \in \NN$ and $\psi(\chi)_i \in \overline{I_E}$; in this case we say that any such $\psi(\chi)$ is \textit{associated} to $\chi$. For any such $\psi(\chi)$ the image of $\cR_\chi$ inside $\cB_\psi$ is an $R$-order which is a complete regular local ring (since it is integral over $R$) which we denote by $\cR_\psi'$.

In particular for all $\chi \in I_E$ we have a decomposition
\begin{align}
	\cR_{\chi}[1/\pi] = \bigoplus_{i=1}^{d_\chi} \cR'_{\psi(\chi)_i}[1/\pi]\label{equivariant-integers-decomposition}
\end{align}
where $\pi$ is a generator of the maximal ideal of $R$. 

To give an example of these constructions we let $H$ be a finite group and set $\cR = R[H]$ with $R$ a finite extension of $\ZZ_p$. Then we have $\cB = Q[H]$ and we may take the $\set{e_\chi}_{\chi \in I_E}$ to be the canonical idempotents associated to each character class $\chi \in \widehat{\cG_E \times H}/\sim_{G_Q}$ admitting a representative of prime-to-$p$ order where $\sim_{Q}$ is the equivalence relation defined by
\begin{align*}
    \chi \sim \chi' \iff \text{there exists } \sigma \in G_{Q} \text{ such that } \chi = \sigma \circ \chi'
\end{align*} 
Similarly, the set $\set{e_\chi}_{\chi \in\overline{I_E}}$ consists of the canonical idempotents associated to each character class $\chi \in \widehat{\cG_E \times H}/\sim_{G_Q}$. Motivated by this example, we shall refer to elements of $I_E$ and $\overline{I_E}$ as `character classes'.

We next remark that $\bLambda_E$ decomposes as
\begin{align}
	\bLambda_E = \bigoplus_{\chi \in I_E} \cR_\chi[[\Gamma]]\label{equivariant-algebra-decomposition}
\end{align}
where we recall that $\Gamma = \Gal(K_\infty/K)$. In particular, if each $\cR_\chi$ is a finite integral extension of $R$ then $\bLambda_E$ is a product of classical Iwasawa algebras and, therefore, a regular ring. In general, $\bLambda_E$ is more complicated but still satisfies sufficiently good ring-theoretic properties. Indeed, it is easily seen from the decomposition (\ref{equivariant-algebra-decomposition}) that $\bLambda_E$ is Gorenstein.

In this setting an important distinction now arises: we say that a prime $\fp \in \Spec(\bLambda_E)$ of height one is \textit{regular} if $p \not\in \fp$ and \textit{singular} otherwise.

Observe that the decompositions (\ref{equivariant-integers-decomposition}) and (\ref{equivariant-algebra-decomposition}) combine to give a decomposition
\begin{align*}
	\bLambda_E\left[\frac{1}{\pi}\right] \cong \bigoplus_{\psi \in \overline{I_E}} \bLambda_{E,\psi}\left[\frac{1}{\pi}\right]
\end{align*}
where $\bLambda_{E,\psi} := \cR_\psi'[[\Gamma]]$. In particular, if $\mathfrak{p}$ is a regular prime of $\bLambda_E$ then there exists a character class $\psi = \psi_\fp \in \overline{I_E}$ such that $\bLambda_{E,\fp}$ coincides with a localisation of $\bLambda_{E,\psi}$ at $\mathfrak{p}\Lambda_\psi$. As such, the localisation $\bLambda_{E,\fp}$ is a discrete valuation ring.

Now suppose that $\fp$ is a singular prime of $\bLambda_E$. Then by the decomposition (\ref{equivariant-algebra-decomposition}) there exists a character class $\chi =\chi_\fp \in I_E$ such that $\bLambda_\fp$ coincides with a localisation of $\bLambda_{E,\chi} := \cR_\chi[[\Gamma]]$ at $\fp\Lambda_\chi$. The following Lemma then shows that the singular primes of $\bLambda_E$ are in one-to-one correspondence with the character classes in $\overline{I_E}$:

\begin{lemma}\label{singular-prime-classification-lemma}
	Let $\chi \in I_E$ be a character class. Then the unique height one prime ideal of $\bLambda_{E,\chi} = \cR_\chi[[\Gamma]]$ containing $p$ is $\sqrt{p\cR_\chi}\bLambda_{E,\chi}$.
\end{lemma}

\begin{proof}
	We first show that the ideal $\sqrt{p\cR_\chi}\bLambda_{E,\chi}$ is indeed a prime ideal. Note that the quotient $\bLambda_{E,\chi}/\sqrt{p\cR_\chi}\bLambda_{E,\chi}$ identifies with $(\cR_\chi/\sqrt{p\cR_\chi})[[\Gamma]]$ and so it suffices to show that this latter ring is a domain.

	To do this we note that $\sqrt{p\cR_\chi}$ coincides with the inverse image of the nilradical of the ring $\cR_\chi/p\cR_\chi$ under the natural projection $\cR_\chi \to \cR_\chi/p\cR_\chi$. Since $p$ is not a unit in $\cR$ the ring $\cR_{\chi}/p\cR_\chi$ is a non-zero local ring of Krull dimension zero and so its nilradical coincides with its unique maximal ideal. From this it follows that $\cR_\chi/\sqrt{p\cR_\chi}$ is a non-zero field of characteristic $p$. This immediately implies that $\cR_\chi/\sqrt{p\cR_\chi}[[\Gamma]]$ is a domain as claimed.

	By definition, $\sqrt{p\cR_\chi}\bLambda_{E,\chi}$ is minimal amongst those primes of $\bLambda_{E,\chi}$ containing the principal ideal $p\Lambda_\chi$. We may then appeal to Krull's principal ideal theorem to establish that $\sqrt{p\cR_\chi}\bLambda_{E,\chi}$ has height one as is required to complete the proof of the Lemma. 
\end{proof}

The above discussion implies that the localisation of $\bLambda_E$ at a singular prime does not enjoy as nice a structure as in the regular case and it is therefore quite cumbersome to study the structure of its category of modules. With this being said, one can circumvent this issue by restricting oneself to $\bLambda_E$-modules for which certain $\mu$-invariants vanish as the following Lemma shows. To state it we will write $\mu_A(M)$ for the $\mu$-invariant of the Iwasawa module $M$ over a given (non-equivariant) Iwasawa algebra $A$.

\begin{lemma}\label{singular-primes-vanishing-lemma}
	Let $M$ be a finitely generated $\bLambda_E$-module which is torsion as a $\Lambda$-module, $\mathfrak{p}$ a singular prime with associated character class $\chi = \chi_\mathfrak{p} \in I_E$. Then the following are equivalent:
	\begin{enumerate}
		\item{The $\mu_\Lambda$-invariant of $e_\chi M$ vanishes,}
		\item{For any character class $\psi = \psi(\chi) \in \overline{I_E}$ associated to $\chi$ the $\mu_{\bLambda_{E,\psi}}$-invariant of
			\begin{align*}
				M_\psi := M \otimes_{\bLambda_E} \bLambda_{E,\psi} \cong e_\chi M \otimes_{\bLambda_E} \bLambda_{E,\psi}
			\end{align*}
			vanishes,
		}
		\item{The $\bLambda_\mathfrak{p}$-module $M_\mathfrak{p}$ vanishes.}
	\end{enumerate} 
\end{lemma}

\begin{proof}
	It is an immediate consequence of the structure theorem for Iwaswa modules that $\mu_{\Lambda}(e_\chi(M))$ (resp. $\mu_{\bLambda_{E,\psi}}(M_\psi)$) vanishes if and only if $(e_\chi M)_{(p)} = 0$ (resp. $M_{\psi,(\pi)}$ with $\pi$ a generator of the maximal ideal of $\cR_\psi'$).

	We now claim that the natural homomorphism
	\begin{align}
		\bLambda_{E,\chi} \otimes_\Lambda \Lambda_{(p)} \cong \bLambda_{E,\chi, \mathfrak{p}}\label{iwasawa-algebra-isomorphic-localisations}
	\end{align}
	of rings is an isomorphism. To do this we fix a topological generator $\gamma$ of $\Gamma$. Then by the argument used to prove Lemma \ref{singular-prime-classification-lemma} the ideal $(\mathfrak{p},\gamma-1)\bLambda_{E,\chi}$ is the unique maximal ideal of the local ring $\bLambda_{E,\chi}$. Observe, furthermore, that $\gamma-1$ is in fact an element of $\Lambda$, not divisible by $p$, and is therefore invertible in $\bLambda_{E,\chi} \otimes_\Lambda \Lambda_{(p)}$. Since the ring $\bLambda_{E,\chi}$ has Krull dimension two it then follows that $\bLambda_{E,\chi} \otimes_\Lambda \Lambda_{(p)}$ has Krull dimension one. A straightforward calculation then shows that the spectrum of $\bLambda_{E,\chi} \otimes_\Lambda \Lambda_{(p)}$ coincides with the subset of $\Spec(\bLambda_{E,\chi})$ consisting of those prime ideals of height one containing $p$. The isomorphism (\ref{iwasawa-algebra-isomorphic-localisations}) is now an immediate consequence of the fact that the only such prime is $\mathfrak{p}$.

	The equivalence between (1) and (3) now follows by tensoring the isomorphism (\ref{iwasawa-algebra-isomorphic-localisations}) over $\bLambda$ with $M$.

	We now note that the map $\cR_\chi \to \cR_\psi'$ is surjective by definition and so the induced map on residue fields $\cR_\chi/\sqrt{p\cR_\chi} \to \cR_\psi'/\pi_\psi\cR_\psi'$ is an isomorphism where $\pi_\psi$ is a generator of the maximal ideal of $\cR_\psi'$. The equivalence between (2) and (3) is now follows by combining this isomorphism with Nakayama's Lemma.
\end{proof}

The following Lemma will allows us to analyse the condition of Proposition \ref{reinterpret-iwasawa-conjecture-proposition} after localising at height one primes. We remark that $\bLambda_E$ is a Cohen-Macaulay ring by virtue of being Gorenstein.

\begin{lemma}\label{height-one-reduction-lemma}
    Let $R$ be a Cohen-Macaulay ring with total ring of quotients $Q$ and suppose $R$ admits a decomposition as a product of local rings. Assume to be given an invertible $Q$-module $M$, a cyclic $R$-submodule $I$ of $M$, and an invertible $R$-submodule $J$ of $M$. Then $I \subseteq J$ if and only if for every height one prime $\fp$ of $R$ one has $I_\fp \subseteq J_\fp$ inside $M_\fp$.
\end{lemma}

\begin{proof}
    This is shown in the same way as \cite[Lem. 5.3]{flach} by using an $R$-generator of $I$ instead of a basis. 
\end{proof}

We end this section by recalling the following useful result on Fitting ideals of $\bLambda_E$-modules:

\begin{theorem}[Greither-Kurihara]\label{greither-kurihara-theorem}
	Fix a field $E \in \Omega(\cK'/K)$ and let $(M_n)_{n \in \NN}$ be an inverse system of $\cR[\cG_{E_n}]$-modules with surjective transition morphisms. If the limit $M_\infty := \varprojlim_{n \in \NN} M_n$ is torsion as a $\Lambda$-module then the natural projection maps $\cR[\cG_{E_m}] \to \cR[\cG_{E_n}]$ induce an isomorphism of Fitting ideals
	\begin{align*}
		\Fitt_\bLambda^0(M_\infty) \cong \varprojlim_{n \in \NN} \Fitt_{\cR[\cG_{E_n}]}^0(M_n)
	\end{align*}
\end{theorem}

\begin{proof}
	This follows via the same argument used to prove \cite[Thm. 2.1]{gk}. The statement of \textit{loc. cit.} is for the case $\cR = R[G]$ where $G$ is a finite abelian group. However, the only point at which this structure is relevant in the argument is in step (5) in which a non-zero-divisor annihilator $f \in \Lambda$ of $M_\infty$ is fixed and the claim that $\bLambda_E/f\bLambda_E$ is a semi-local ring of dimension one is asserted. Since it is clear that this condition is also satisfied under the slightly more general hypotheses levied on $\cR$ herein, the argument of \textit{loc. cit.} carries through unchanged.
\end{proof}

\subsubsection{Iwasawa-theoretic divisibilities}

In this section we expand upon Remark \ref{eimc-remark} by giving the precise statement and proof of Theorem \ref{basic-criterion-theorem}. In particular, we shall show that the question of whether an Euler system is basic is equivalent (modulo vanishing of particular $\mu$-invariants) to a collection of classical Iwasawa-theoretic divisibilities.

To do this we first record the following Lemma:

\begin{lemma}\label{ES-torsion-lemma}
    Fix an Euler system $\eta \in \ES_{r}(T,\cK)$, a field $E \in \Omega(\cK'/K)$, and a regular prime $\fp$ of $\bLambda_E$. If $\braket{(\eta_{E_n})_n}_{\bLambda_{E,\fp}}$ is non-trivial then
    \begin{align*}
        \frac{\displaystyle\Bigl({\textstyle\varprojlim_{n}} \bidual_{\cR[\cG_{E_n}]}^{r} H^1(\cO_{E_n, S(E_n)}, T)\Bigr)_\fp}{\displaystyle\braket{(\eta_{E_n})_n}_{\bLambda_{E,\fp}}}
    \end{align*}
    is a torsion $\bLambda_{E,\fp}$-module.
\end{lemma}

\begin{proof}
    At the outset we fix a quadratic standard representative $[P \xrightarrow{\psi} P]$ of the complex $C_{E_\infty, S(E)}(T)$, the existence of which is guaranteed by Lemma \ref{bidual-compatibility-lemma}. 
    
	We next recall that $\bLambda_{E,\fp}$ is a discrete valuation ring and, therefore, a principal ideal domain. The embedding $H^1_\Iw(\cO_{E,S(E)},T) \subseteq P$ then implies that $H^1_\Iw(\cO_{E,S(E)}, T)_\fp$ is a free $\bLambda_{E,\fp}$-module, being a submodule of a finite free module. A straightforward analysis of the Yoneda 2-extension associated to the fixed standard representative of the complex $C_{E_\infty, S(E)}(T)$ then implies that the $\bLambda_{E,\fp}$-rank of this module is precisely $r$.

	In particular one sees that
	\begin{align*}
		\left(\bidual_{\bLambda_E}^{r} H^1_\Iw(\cO_{E,S(E)}, T)\right)_\fp \cong \bidual_{\bLambda_{E,\fp}}^{r} H^1_{\Iw}(\cO_{E,S(E)}, T)_\fp \cong \exprod_{\bLambda_{E,\fp}}^{r} H^1_\Iw(\cO_{E,S(E)}, T)_\fp
	\end{align*}
	is $\bLambda_{E,\fp}$-free of rank one.

	As such the assumption that the module $\braket{(\eta_{E_n})_n}_{\bLambda_{E,\fp}}$ be non-trivial means that it is also free of rank one and, therefore, that the quotient of the Lemma is $\bLambda_{E,\fp}$-torsion.
\end{proof}

\begin{remark}
	Let $E \in \Omega(\cK'/K)$ be a field and $\psi \in \overline{I_E}$ a character class. Then by fixing any regular height one prime $\fp$ of $\bLambda_{E,\psi}$ and regarding it as a prime of $\bLambda_{E}$ one may immediately deduce from Lemma \ref{ES-torsion-lemma} the following.

	Let $\eta \in \ES_{r}(T,\cK)$ be an Euler system. If there is some $n \in \NN$ such that $\eta_{E_n}^\psi$ is non-trivial then
    \begin{align*}
        \frac{\displaystyle\Bigl({\textstyle\varprojlim_{n}} \bidual_{\cR[\cG_{E_n}]}^{r_T} H^1(\cO_{E_n, S(E_n)}, T)\Bigr)^\psi}{\displaystyle\braket{(\eta_{E_n})_n^\psi}_{\bLambda_{E,\psi}}}
    \end{align*}
    is a torsion $\bLambda_{E,\psi}$-module where we write $(-)^\psi$ for the functor $- \otimes_{\bLambda_E} \bLambda_{E,\psi}$.
\end{remark}

\begin{theorem}\label{explicit-conditions-for-conj-theorem}
    Suppose to be given an Euler system $\eta \in \ES_r(T,\cK)$. Assume that for each field $E \in \Omega(\cK'/K)$ and every character class $\chi \in I_E$ one has that the $\mu_\Lambda$-invariant of
    \begin{align}
       e_\chi \cdot H^2_\Iw(\cO_{E,S(E)}, T)\label{mu-invariant-condition}
    \end{align}
    vanishes. Then the following are equivalent
    \begin{enumerate}
        \item{$\eta \in \ES^b(T,\cK)$;}
        \item{For every field $E \in \Omega(\cK'/K)$ and regular prime $\fp$ of $\bLambda_E$ such that $\braket{(\eta_{E_n})_n}_{\bLambda_{E,\fp}}$ is non-trivial there is an inclusion
        \begin{align}\label{theorem-inclusion-3}
            \Fitt_{\bLambda_E}^0\left(\frac{\displaystyle\varprojlim_{n \in \NN} \bidual_{\cR[\cG_{E_n}]}^r H^1(\cO_{E_n, S(E_n)}, T)}{\displaystyle\braket{(\eta_{E_n})_n}_{\bLambda_E}}\right)_\fp \subseteq \Fitt^0_{\bLambda_E}\left(H^2_\Iw(\cO_{E, S(E)}, T)\right)_\fp
        \end{align}
        }
        \item{For every field $E \in \Omega(\cK'/K)$ and character class $\psi \in \overline{I_E}$ for which there exists an $n \in \NN$ such that $e_\psi\cdot\eta_{E_n}$ is non-trivial one has that
        \begin{align}
            \begin{array}{c}
                \displaystyle\mathrm{char}_{\bLambda_{E,\psi}}\left(H^2_\Iw(\cO_{E,S(E)}, T)^\psi\right) \\\\
                divides\\\\
                \mathrm{char}_{\bLambda_{E,\psi}}\left(\frac{\displaystyle\left(\varprojlim_{n \in \NN} \bidual_{\cR[\cG_{E_n}]}^r H^1(\cO_{E_n, S(E_n)}, T)\right)^\psi}{\displaystyle\braket{(\eta_{E_n})_n}^\chi_{\bLambda_{E,\psi}}}\right)
            \end{array}\label{iwasawa-theoretic-divisibilities}
        \end{align}
        where we write $(-)^\psi$ for the functor $- \otimes_{\cR[\cG_E]} \cR_\psi = - \otimes_{\bLambda_E} \bLambda_{E,\psi}$.}
    \end{enumerate}
    Moreover, suppose that for each field $E \in \Omega(\cK'/K)$ and every character class $\chi \in I_E$ one has that the $\mu_\Lambda$-invariant of
    \begin{align*}
        \frac{\displaystyle\left(\varprojlim_{n \in \NN} \bidual_{\cR[\cG_{E_n}]}^r H^1(\cO_{E_n, S(E_n)}, T)\right)^\chi}{\displaystyle\braket{(\eta_{E_n})_n}^\chi_{\bLambda_{E}}}
    \end{align*}
    vanishes. Then $\eta$ is an $\cR[[\Gal(\cK/K)]]$-basis of $\ES^b(T,\cK)$ if and only if equality holds in either (and thus both) (2) and (3).
\end{theorem}

\begin{proof}
    We shall establish these equivalences by arguing that $(1) \iff (2)$ and $(2) \iff (3)$
    
    At the outset we remark that by Proposition \ref{reinterpret-iwasawa-conjecture-proposition} one knows that $\eta \in \ES^b(T,\cK)$ if and only if for every $E \in \Omega(\cK/K)$ one has an inclusion
    \begin{align}
        \cL_{\eta,E} \in \Det_{\bLambda_E}(C_{E_\infty, S(E)}(T))\label{theorem-inclusion}
    \end{align}

    Now fix $E \in \Omega(\cK'/K)$. Without loss of generality we may assume that there exists $n \in \NN$ such that $\eta_{E_n}$ is non-trivial. Indeed if $\eta_{E_n}$ is trivial for all $n \in \NN$ then the inclusion (\ref{theorem-inclusion}) is vacuous and so there is nothing to prove in this case. 
    
    Appealing to Lemma \ref{height-one-reduction-lemma} we see that it suffices to check the above inclusion after localisation at a height one prime of $\bLambda_E$.

    We first perform the analysis for the singular primes of $\bLambda_E$. Fix such a prime $\fp$. By Lemma \ref{singular-primes-vanishing-lemma} and the hypothesis, one knows that $(H^2_\Iw(\cO_{E,S(E)}, T)_\fp = 0$. This combines with the exact sequence (\ref{iwasawa-cohomology-exact-sequence}) and the vanishing of the Euler-Poincar\'e characteristic of the relevant complex to imply that $(H^1_\Iw(\cO_{E,S(E)}, T))_\fp$ is free of rank $r$. We may thus pass to cohomology and apply the injective map $\Theta_{T,E, \fp}^\infty$ to find that the inclusion (\ref{theorem-inclusion}) is equivalent to the inclusion
    \begin{align}
        \braket{(\eta_{E_n})_n}_{\bLambda_{E,\fp}} \subseteq \left(\varprojlim_{n \in \NN} \bidual_{\cR[\cG_{E_n}]}^r H^1(\cO_{E_n,S(E_n)}, T)\right)_\fp\label{singular-inclusion}
    \end{align}
    
    We next analyse both directions of the implication for regular primes. To this end, fix first a regular prime $\fp$ of $\bLambda_E$. Then $\bLambda_{E,\fp}$ is a discrete valuation ring and is thus regular. In particular, we may pass to cohomology in the determinant functor. Moreover, $\bLambda_{E,\fp}$ has dimension one and so the determinants of finitely generated torsion $\bLambda_{E,\fp}$-modules coincide with the inverse of their initial Fitting ideals. These two facts taken together imply that the desired inclusion
    \begin{align}
        \braket{\cL_{\eta,E}}_{\bLambda_{E,\fp}} \subseteq \Det_{\bLambda_{E,\fp}}(\bLambda_{E,\fp} \otimes_{\bLambda_{E}} C_{E_\infty, S(E)}(T))\label{theorem-inclusion-2}
    \end{align}
    is equivalent to an inclusion
    \begin{align*}
        \braket{(\eta_{E_n})_n}_{\bLambda_{E,\fp}} \subseteq \Fitt^0_{\bLambda_{E}}\left(H^2_\Iw(\cO_{E, S(E)}, T)\right)_\fp\cdot \left(\varprojlim_{n} \bidual_{\cR[\cG_{E_n}]}^r H^1(\cO_{E_n, S(E_n)}, T)\right)_\fp
    \end{align*}
    If $\braket{(\eta_{E_n})_n}_{\bLambda_{E,\fp}}$ is trivial then this inclusion is vacuous. Otherwise, Lemma \ref{ES-torsion-lemma} implies that this inclusion is equivalent to the inclusion
    \begin{align}\label{fitting-ideal-inclusion}
        \Fitt_{\bLambda_{E}}^0\left(\frac{\displaystyle\varprojlim_{n \in \NN} \bidual_{\cR[\cG_{E_n}]}^r H^1(\cO_{E_n, S(E_n)}, T)}{\displaystyle\braket{(\eta_{E_n})_n}_{\bLambda_E}}\right)_\fp \subseteq \Fitt^0_{\bLambda_E}\left(H^2_\Iw(\cO_{E, S(E)}, T)\right)_\fp
    \end{align}
    This establishes the equivalence $(1) \iff (2)$ of the Theorem.
    
    Now suppose that the divisibility (\ref{iwasawa-theoretic-divisibilities}) holds for all $\psi \in \overline{I_E}$ such that there exists an $n \in \NN$ with $e_\psi\cdot\eta_{E_n}$ non-trivial. Fix a regular prime $\fp$ of $\bLambda_E$. Then there exists $\psi \in \overline{I_E}$ such that $\psi_\fp = \psi$. Observe that the localisation homomorphism $\bLambda_E \to \bLambda_{E,\fp}$ factors through the map $\bLambda_E \to \bLambda_{E,\psi}$ induced by multiplication by $e_\chi$. Moreover for every finitely generated torsion $\bLambda_{E,\chi}$-module $M$, \cite[Lem. 3.4.2]{nguyen-quang-do-nicolas} gives an equality
    \begin{align}\label{fitt-equals-char}
        \Fitt_{\bLambda_{E,\psi}}^0(M) = \Fitt_{\bLambda_{E,\psi}}^0(M_\mathrm{fin})\cdot \mathrm{char}_{\bLambda_{E,\psi}}(M)
    \end{align}
    where $M_\mathrm{fin}$ is the maximal finite submodule of $M$. Since $\fp$ is regular we may therefore base-change the condition (\ref{iwasawa-theoretic-divisibilities}) to $\bLambda_\fp$ to establish the implication $(3) \implies (2)$.
    
   Conversely, observe that the definition of characteristic ideals implies that for finitely generated torsion $\bLambda_{E,\psi}$-modules $M$ and $N$ one has that $\mathrm{char}(M) \subseteq \mathrm{char}(N)$ if and only if $\mathrm{char}(M_\fp) \subseteq \mathrm{char}(N_\fp)$ for all height one primes $\fp$ of $\bLambda_{E,\psi}$. Given this, the implication $(2) \implies (3)$ is now an immediate consequence of the equality (\ref{fitt-equals-char}) and the $\mu$-vanishing assumption.

    Finally, to prove the claim regarding bases of $\ES^b(T,\cK)$ we first observe that the second $\mu$-vanishing condition in the statement is equivalent to the assertion that the inclusion (\ref{singular-inclusion}) is in fact an equality. The claim then follows from the same argumentation as above by replacing all relevant inclusions/divisibilities with equalities and noting that $\eta$ is a basis of $\ES^b(T,\cK)$ if and only if for every $E \in \Omega(\cK'/K)$ one has that $\cL_{\eta,E}$ is a basis of $\Det_{\bLambda_E}(C_{E_\infty, S(E)}(T))$.
    
    This completes the proof of the Theorem.
\end{proof}

\begin{remark}
    Euler systems are widely believed to satisfy divisibility conditions of the shape (\ref{iwasawa-theoretic-divisibilities}). Such conditions are well-known in the rank one setting and appear in, for example, \cite[Thm. 3.3]{rubin}. They are also known to hold in certain cases in the higher rank setting for the Euler system of Rubin-Stark elements and in this regard we refer the reader to \cite[Thm. A]{buyukboduk}. As such, Theorem \ref{explicit-conditions-for-conj-theorem} can be seen to give evidence for Question \ref{main-question}.
\end{remark}

\begin{remark}\label{mu-vanishing-remark}
    The $\mu$-vanishing condition (\ref{mu-invariant-condition}) has been studied by various authors for the cyclotomic $\ZZ_p$-extension $E_\infty/E$ and for varying $T$:
    \begin{enumerate}
        \item{For $T = \ZZ_p(1)$ this is the celebrated conjecture of Iwasawa which is still only known in a few cases such as when $E$ is abelian over $\QQ$ through the work of Ferrero-Washington \cite{ferrero-washington}}.
        \item{If $T$ is the $p$-adic Tate module of an elliptic curve then Coates and Sujatha have conjectured in \cite[Conj. A]{coates-sujatha} that the dual of the fine Selmer group of $T$ over $E_\infty$ has vanishing $\mu$-invariant (and see below for the equivalence with Iwasawa cohomology).}
        \item{In general, Lim has conjectured in \cite[Conj. A]{lim} that the dual of the fine Selmer group of $T$ over $E_\infty$ has vanishing $\mu$-invariant. We remark that by Lemma 3.4 of \textit{loc. cit.} this is equivalent to the vanishing of the $\mu$-invariant of $H^2_\Iw(\cO_{E,S(E)}, T)$. Moreover, in Theorem 3.1 of \textit{loc. cit.} it is shown that if Iwasawa's conjecture holds for a particular finite extension $F$ of $E$ then Lim's conjecture holds for $E_\infty$.}
    \end{enumerate}
\end{remark}

\subsection{Standard hypotheses revisited}

In this section we consider representations and fields satisfying standard hypotheses in order to obtain stronger evidence in favour of Question \ref{main-question}. As such, we shall make use of the notations and assumptions established in \S\ref{standard-hypotheses-section}. We assume, for simplicity, that $K_\infty$ is the cyclotomic $\ZZ_p$-extension of $K$. We remark, however, that by considering Iwasawa-theoretic versions of $\mathrm{(SH_2)}$ and $\mathrm{(SH_3)}$ one can prove analogous results for other $\ZZ_p$-extensions (see \cite[Th. 4.27, Rem. 4.28]{bss2}).

In the sequel we write $\cL$ for the maximal subextension of $\cK'$ such that every field in $\Omega(\cL K_\infty/K)$ satisfies standard hypotheses for the representation $T$.

\begin{proposition}\label{iwasawa-inclusion-holds-proposition}
    Let $\eta \in \ES_r(T,\cK)$ be an Euler system. Then for every field $E \in \Omega(\cL/K)$ and regular prime $\fp$ of $\bLambda_E$ for which $\braket{(\eta_{E_n})_n}_\fp$ is non-trivial the inclusion (\ref{theorem-inclusion-3}) holds.
\end{proposition}

\begin{proof}
    Fix a field $E \in \Omega(\cL/K)$. Then it is easy to see that for every $n \in \NN$, the field $E_n$ satisfies standard hypotheses for the representation $T$. As such, \cite[Lem. 3.10]{bss2} implies that the canonical and relaxed Selmer structures coincide for the induced representations $\cT_{E_n}$. 
    
    After taking into account \cite[Ex. 2.7(iii)]{bss2} we then have a canonical isomorphism
    \begin{align*}
        H^1_{\cF_{\mathrm{can}}^*}(K, \cT_{E_n}^\lor(1))^\lor \cong H^2(\cO_{E_n, S(E_n)}, T)
    \end{align*}
    We may thus apply \cite[Thm. 3.6(iii)]{bss2} to deduce an inclusion
    \begin{align*}
        \im(\eta_{E_n}) \subseteq \Fitt_{\cR[\cG_{E_n}]}^0(H^2(\cO_{E_n, S(E_n)}, T))
    \end{align*}
    for every $n \in \NN$.
    
    Now fix a regular prime $\fp$ of $\bLambda_E$. By Lemma \ref{image-containment-lemma} below we can pass to the limit to obtain
    \begin{align}\label{iwasawa-theoretic-inclusion}
        \im((\eta_{E_n})_n)_\fp \subseteq \Fitt_{\bLambda_E}^0(H^2_\Iw(\cO_{E, S(E)}, T))_\fp
    \end{align}
    Consider next the tautological exact sequence
    \begin{align*}
        0 \to \braket{(\eta_{E_n})_n}_{\bLambda_{E,\fp}} \to \bidual_{\bLambda_{E,\fp}}^r H^1_{\Iw}(\cO_{E, S(E)}, T)_\fp \to Q \to 0
    \end{align*}
    Since $\bLambda_{E,\fp}$ is a discrete valuation ring, a straightforward analysis of the Yoneda 2-extension associated to the representative $[P_\fp \to P_\fp]$ arising from Lemma \ref{bidual-compatibility-lemma}(2) implies that the first two terms of this sequence are free $\bLambda_{E,\fp}$-modules of rank one.
    
    Given this, we may apply \cite[Lem. A.2(ii)]{bs} to deduce that
    \begin{align*}
        \im((\eta_{E_n})_n)_\fp = \Fitt_{\bLambda_{E,\fp}}^0(Q)
    \end{align*}
    The desired inclusion (\ref{theorem-inclusion-3}) now follows immediately by combining the inclusion (\ref{iwasawa-theoretic-inclusion}) with the identification
    \begin{align*}
        Q \cong \frac{\displaystyle\varprojlim_{n \in \NN} \bidual_{\cR[\cG_{E_n}]}^r H^1(\cO_{E_n, S(E_n)}, T)}{\braket{(\eta_{E_n})_n}_{\bLambda_E}}
    \end{align*}
    induced by the isomorphism of Lemma \ref{bidual-compatibility-lemma}.
\end{proof}

\begin{lemma}\label{image-containment-lemma}
    Fix a field $E \in \Omega(\cK'/K)$.
    \begin{enumerate}
        \item{There is an equality of ideals
            \begin{align*}
                \varprojlim_{n \in \NN} \Fitt_{\cR[\cG_{E_n}]}^0(H^2(\cO_{E_n, S(E_n)}, T)) = \Fitt_{\bLambda_E}^0(H^2_\Iw(\cO_{E,S(E)}, T)) 
            \end{align*}
        }
        \item{For every Euler system $\eta \in \ES_r(T,\cK)$ there is an inclusion of ideals
        \begin{align*}
            \im((\eta_{E_n})_n) \subseteq (\varprojlim_n \im(\eta_{E_n}))
        \end{align*}
        with pseudo-null cokernel where $(\eta_{E_n})_n$ is viewed as an element of $\bidual_{\bLambda_E}^r H^1_\Iw(\cO_{E,S(E)}, T)$ via the isomorphism of Lemma \ref{bidual-compatibility-lemma}.
        }
    \end{enumerate}
\end{lemma}

\begin{proof}
    By \cite[Lem. 4.29]{bss2} one knows that the restriction maps $\cR[\cG_{E_m}] \to \cR[\cG_{E_n}]$ for $m > n$ induce surjective maps
    \begin{align*}
        \Fitt_{\cR[\cG_m]}^0(H^2(\cO_{E_m, S(E_m)}, T)) \twoheadrightarrow \Fitt_{\cR[\cG_n]}^0(H^2(\cO_{E_n, S(E_n)}, T)), \quad \im(\eta_{E_m}) \twoheadrightarrow \im(\eta_{E_n})
    \end{align*}
    By passing to the limit over the first of these transition maps and applying Theorem \ref{greither-kurihara-theorem} one obtains the identification in the first claim of the Lemma.
    
    As for the second claim, we first observe that the surjectivity of the second of the above transition maps implies that $\varprojlim_n \im(\eta_{E_n})$ is a well-defined non-zero ideal of $\bLambda_E$.
    
    Now, the definition of the bidual implies that there is an equality of ideals
    \begin{align*}
        \im((\eta_{E_n})_n) = \Set{\phi((\eta_{E_n})_n) | \phi \in \exprod_{\bLambda_E}^r H^1_\Iw(\cO_{E,S(E)}, T)^*}
    \end{align*}
    By Lemma \ref{bidual-compatibility-lemma}(2) there exists a short exact sequence
    \begin{align*}
        0 \to H^1_\Iw(\cO_{E,S(E)}, T) \to P \to I \to 0
    \end{align*}
    where $P$ is a free $\bLambda_E$-module and $I$ is $\bLambda_E$-torsion-free. Appealing to \cite[Prop. 5.4.17, Cor. 5.5.4(ii)]{NSW} one sees that $\Ext^1_{\bLambda_E}(I,\bLambda_E)$ is finite and so the natural map
    \begin{align*}
        P^* \to H^1_\Iw(\cO_{E,S(E)}, T)^*
    \end{align*}
    has pseudo-null cokernel. Given this we have an inclusion of ideals
    \begin{align*}
        \im((\eta_{E_n})_n) \subseteq \Set{\phi((\eta_{E_n})_n) | \phi \in \exprod_{\bLambda_E}^r P^*}
    \end{align*}
    with pseudo-null cokernel.
    
    Moreover, if we set $P_n := P \otimes_{\bLambda_E} \cR[\cG_{E_n}]$ then an entirely similar argument applied to the representative $P_n \to P_n$ of $C_{E_n, S(E_n)}(T)$ gives the equality
    \begin{align*}
        \im(\eta_{E_n}) = \Set{\phi(\eta_{E_n}) | \phi \in \exprod_{\cR[\cG_{E_n}]}^r P_n^*}
    \end{align*}
    
    On the other hand, since $P$ is projective we obtain a composite isomorphism
    \begin{align*}
        \Hom_{\bLambda_E}\left(\exprod_{\bLambda_E}^r P, \bLambda_E\right)
         &\cong \Hom_{\bLambda_E}\left(\varprojlim_n \exprod_{\cR[\cG_{E_n}]}^r P_n, \bLambda_E\right)\\
         &\cong \varprojlim_m \Hom_{\bLambda_E}\left(\varprojlim_n \exprod_{\cR[\cG_{E_n}]}^r P_n, \cR[\cG_{E_m}]\right)\\
        &\cong \varprojlim_m \Hom_{\cR[\cG_{E_m}]}\left(\exprod_{\cR[\cG_{E_m}]}^r P_m, \cR[\cG_{E_m}]\right)
    \end{align*}
    The Lemma now follows by combining this isomorphism with the above descriptions of the relevant ideals. 
\end{proof}

We end this section with a straightforward Corollary of Proposition \ref{iwasawa-inclusion-holds-proposition}.

\begin{corollary}\label{iwasawa-inclusion-holds-corollary}
    Suppose that for every field $E \in \Omega(\cL/K)$ the $\mu$-invariant hypothesis (\ref{mu-invariant-condition}) holds. Then for every Euler system $\eta \in \ES_r(T,\cK)$ there exists a basic Euler system that agrees with $\eta$ on the extension $\cL K_\infty/K$.
\end{corollary}

\begin{proof}
    Fix an Euler system $\eta \in \ES_r(T,\cK)$. Then by Proposition \ref{iwasawa-inclusion-holds-proposition} and Theorem \ref{explicit-conditions-for-conj-theorem} one knows that for every $E \in \Omega(\cL/K)$ the element $\cL_{\eta, E}$ is contained in $\Det_{\bLambda_E}(C_{E_\infty, S(E)}(T))$.
    
    Now since the transition map $\nu_{E'/E}^\infty$ in the definition of $\VS_\Iw(T,\cK)$ is surjective, the natural projection map
    \begin{align*}
        \VS_\Iw(T,\cK) \twoheadrightarrow \VS_\Iw(T,\cL K_\infty)
    \end{align*}
    is a surjection. We are therefore able to lift the vertical system $(\cL_{\eta,E})_{E \in \Omega(\cL/K)}$ to a vertical system $z \in \VS_\Iw(T,\cK)$. The element $\Theta_{T,\cK}^\infty(z)$ then gives the desired basic Euler system.
\end{proof}

\section{Applications to arithmetic}

\subsection{The multiplicative group}

\subsubsection{The set-up}

Let $K$ be a totally real field and $\psi: G_K \to \overline{\QQ_p}^\times$ a non-trivial character of finite prime-to-$p$ order. Denote by $\cR$ the finite $\ZZ_p$-algebra generated by the image of $\psi$. If $L$ denotes the fixed field of $\psi$ then we suppose throughout this section that $L$ is totally real and that no $p$-adic place of $K$ splits completely or ramifies in $L$. We moreover write $\Delta = \Gal(L/K)$.

We write $\cR(1) \otimes \psi^{-1}$ for the $G_K$-representation which equals $\cR$ as an $\cR$-module and upon which $G_K$ acts via $\chi_\mathrm{cyc}\psi^{-1}$ where $\chi_\mathrm{cyc}$ is the cyclotomic character of $K$.

Then $T = \cR(1) \otimes \psi^{-1}$ is a $p$-adic representation with coefficients in $\cR$. Fix a finite set of places $S$ of $K$ containing $S_\mathrm{min}(T)$ and an abelian extension $\cK/K$ such that the pair $(T,\cK)$ is admissible.

In this setting we note that the weak Leopoldt conjecture holds by the result recalled in Remark \ref{weak-leopolodt-remark}.

For each $E \in \Omega(\cK/K)$ we make the following definitions:
\begin{itemize}
    \item{$U_{E,S(E)} := (\cO_{LE,S(E)}^{\times})^\psi$ and $U_{E} := (\cO_{LE}^\times)^\psi$}
    \item{$A_{E,S(E)} := \Cl_{S(E)}(LE)^\psi$ and $A_{E} := \Cl(LE)^\psi$}
    \item{$Y_{E,V}$ for the $(p,\psi)$-part of the free abelian group on the places in $V_{LE}$ where $V$ is any finite set of places of $K$.}
\end{itemize}
For any $\square \in \set{U, A, Y}$ we set $\square_{E,V}^\infty := \varprojlim_n \square_{E_n, V}$. Here the limit is taken with respect to the homomorphisms induced by the norm map if $\square \in \set{U,A}$ and with respect to the map induced by restriction of places otherwise.

Note, moreover, that for any pair of fields $E \subseteq E'$ there is another transition map
\begin{align*}
	Y_{E',V} &\to Y_{E,V}\\
	w &\mapsto f_{w/v}v
\end{align*}
where for any place $w$ of $V_{E'}$ lying above $v \in V_E$ we write $f_{w/v}$ for the inertial degree. We then denote by $Y^{\overline{\infty}}_{E,V}$ the limit $\varprojlim_{n \in \NN} Y_{E_n,V}$ taken with respect to this map.

Observe that Kummer theory gives a canonical isomorphism
\begin{align*}
    U_{E,S(E)} \cong H^1(\cO_{E,S(E)}, T)
\end{align*}
In particular, any $\eta \in \ES_r(T,\cK)$ can be regarded as a compatible collection of elements in $U_{E,S(E)}$ as $E$ varies over $\Omega(\cK/K)$.

We first record the following Lemma.

\begin{lemma}\label{Gm-lemma}
    Let $K = \QQ$ and $\eta \in \ES_1(T,\cK)$ be an Euler system and $E \in \Omega(\cK'/K)$. Then for every regular prime $\fp$ of $\bLambda_E$ the condition (\ref{theorem-inclusion-3}) is equivalent to an inclusion
    \begin{align}\label{gm-lemma-inclusion}
        \Fitt^0_{\bLambda_{E,\fp}}((U_{E}^\infty/(\eta_{E_n})_n)_\fp)\cdot \Fitt_{\bLambda_{E,\fp}}^0(Y_{E,S(E)\setminus S_p(K), \fp}^{{\infty}})^{-1} \subseteq \Fitt^0_{\bLambda_{E,\fp}}(A_{E,\fp}^\infty)
    \end{align}
\end{lemma}

\begin{proof}
    At the outset we remark that we have a natural exact sequence
    \begin{align*}
        0 \to U_{E}^\infty \to U_{E, S(E)}^\infty \to Y_{E,S(E)}^{\overline{\infty}} \to A_E^\infty \to A_{E,S(E)}^\infty \to 0
    \end{align*}
    Since no finite place of $LE$ can split completely in $LE_\infty$ it follows that $Y_{E, S(E)}^{\overline{\infty}} = Y_{E, S_p(K)}^{\overline{\infty}}$. By hypothesis, no $p$-adic place of $K$ can split completely in $L$ and so this latter limit is in fact trivial.
    
    As such, the inclusion and quotient maps respectively give canonical identifications $U_E^\infty = U_{E,S(E)}^\infty$ and $A_E^\infty \cong A_{E,S(E)}^\infty$.
    
    On the other hand, the fact that $\psi$ is non-trivial implies that we have an exact sequence
    \begin{align}\label{H2Iw-exact-sequence}
        0 \to A_E^\infty \to H^2_\Iw(\cO_{E,S(E)}, T) \to Y_{E,S(E)\setminus S_p(K)}^{{\infty}} \to 0
    \end{align}
    
    The proof of the Lemma now follows from the well-known fact that Fitting ideals are multiplicative on short-exact sequences over rings of global dimension one. 
\end{proof}

\subsubsection{An interlude on Euler systems in the sense of Kolyvagin}\label{interlude-section}

A key step in the proof of Theorem \ref{pre-question-valid-theorem} will be to use the argument employed by Greither \cite[Thm. 3.1]{greither} in proving the Iwasawa main conjecture over arbitrary abelian extensions. The argument of \textit{loc. cit.} is carried out in the classical framework of Euler systems in the sense of Kolyvagin, the notation of which differs from the more modern definition \ref{euler-system-definition}. In order to facilitate exposition in the sequel we now recall this earlier definition and how elements of $\ES_1(T,\cK)$ give rise to such systems.

To do this we follow Greither in \cite[p. 471]{greither} (and Burns and Seo in \cite[\S 3.2]{BuSe}) by fixing a rational prime $p$, an abelian field $F \in \Omega(K^\ab/K)$ and $M$ a power of $p$. We let $\set{\zeta_n}_{n \geq 1}$ be a compatible family of primitive $n^{th}$-roots of unity and we denote $F(m) = F(\zeta_m)$ for each natural number $m$. We write $\cJ = \cJ_{F,M}$ for the set of square-free positive integers whose every prime divisor both splits completely in $F$ and is congruent to 1 modulo $M$.

\begin{definition}\label{euler-kolyvagin-system-definition}
	A \textit{$p$-adic Euler-Kolyvagin system for the field $F$} is a collection $(\varepsilon_{r})_{r \in \cJ}$ satisfying the following properties for a given $r \in \cJ_{F,M}$ and each prime divisor $\ell$ of $r$:
	\begin{enumerate}[labelwidth=1.5em,labelindent=10em,leftmargin=4.2em]
		\item[$(\ES_1)\;\,$]{$\varepsilon_r$ belongs to $F(r)^\times \widehat{\otimes}_\ZZ \ZZ_p$,}
		\item[$(\ES_2)\;\,$]{$\varepsilon_r$ belongs to $\cO_{F(r),p}^\times$ when $r > 1$,}
		\item[$(\ES_3)\;\,$]{$N_{F(r)/F(r/\ell)}(\varepsilon_{r}) = (1-\Fr_\ell^{-1})\cdot\varepsilon_{r/\ell}$,}
		\item[$(\ES_4)_p$]{For all primes $\mathfrak{l}$ of $F(r)$ lying above $\ell$ we have $\varepsilon_r \equiv \varepsilon_{r/\ell} \pmod{\mathfrak{l}}_p$ where $(\mathrm{mod}\;\mathfrak{l})_p$ is taken to mean that the congruence holds only after projecting to the $p$-primary component of the residue field $\cO_{F(r)}/\mathfrak{l}$.}
	\end{enumerate}
\end{definition}

Henceforth we shall assume that $K = \QQ$, $\cK'$ is the maximal abelian pro-$p$ extension of $K$ unramified at the primes in $S$ and $\cK = \cK'\QQ_\infty$ where $\QQ_\infty$ is the cyclotomic $\ZZ_p$-extension of $\QQ$.

In the following Lemma we shall demonstrate how elements of $\ES_1(T,\cK)$ give rise to useful Euler-Kolyvagin systems for fields $E \in \Omega(\cK/K)$. While this is in principle a straightforward task, there are two minor inconveniences that make matters somewhat delicate. Firstly, elements of $\eta \in \ES_1(T,\cK)$ do not admit values over full ray class fields, rather only over the maximal $p$-subextensions of the various cyclotomic fields over $L$. Secondly, $\eta$ admits values inside the $\psi$-isotypical part of the various unit groups. We will deal with these issues by introducing certain useful splittings of modules of Euler systems.

To formulate this we write $\mathfrak{f}_E$ (resp. $E^p$) for the conductor (resp. maximal $p$-subextension) of an abelian field $E$ and $N^\dag$ for the set of natural numbers congruent to 2 modulo 4.

\begin{proposition}\label{euler-kolyvagin-proposition}
    Fix a field $E \in \Omega(\cK/K)$ and a natural number $d \in \NN^\dag$ such that $d \mid \mathfrak{f}_E$. Given an Euler system $\eta \in \ES_1(T,\cK)$ there exists a $p$-adic Euler-Kolyvagin system $(\varepsilon_{\eta,d,r})_{r \in \cJ_{F,M}}$ for the field $LE$ with the property that
	\begin{align*}
		\varepsilon_{\eta,d,1} = N_{L\QQ(\mathfrak{f}_{E})^p/LE}(\eta_{\QQ(d)^p})
	\end{align*}
	In particular, if $E = \QQ(n)^p$ then $\varepsilon_{\eta,d,1} = \eta_{\QQ(d)^p}$.
\end{proposition}

\begin{proof}
    At the outset we shall first demonstrate the existence of a $\cR_{L\cL/\QQ}$-splitting of the natural map
    \begin{align*}
        \Psi: \ES_1(\ZZ_p(1), L\cL/\QQ) \to \ES_1(\ZZ_p(1), L\cK/\QQ) \to \ES_1(T,\cK/\QQ)
    \end{align*}
    Here $\cL$ is the extension $\cL'\QQ_\infty$ with $\cL'$ the maximal abelian extension of $\QQ$ unramified outside $S$.

To this end we fix a field $E \in \Omega(\cK/\QQ)$ and remark that the fact that $|\Delta|$ is prime-to-$p$ implies the existence of a canonical decomposition
    \begin{align*}
        \ZZ_p[\cG_{LE}] \cong \bigoplus_{\phi \in \widehat{\Delta}} \ZZ_p(\im(\phi))[\cG_E]
    \end{align*}
    By tensoring this decomposition over $\ZZ_p[\cG_{LE}]$ with $\cO_{LE}^\times$ we then obtain the decomposition
    \begin{align*}
        \cO_{LE}^\times \cong \bigoplus_{\phi \in \widehat{\Delta}} \cO_{LE}^{\times,\phi}
    \end{align*}
    In particular one sees that there is an embedding
    \begin{align}
        \iota_E: U_{E} \hookrightarrow \cO_{LE}^\times
    \end{align}
    with the property that the composite $U_{E} \hookrightarrow \cO_{LE}^\times \to U_E$ is the identity and, for each pair of fields $E \subseteq E'$, the diagram
    \begin{center}
        \begin{tikzcd}
            U_{E'} \arrow[r, hookrightarrow] \arrow[d, "N_{E'/E}"] &\cO_{LE',S(E')} \arrow[d, "N_{E'/E}"]\\
            U_E \arrow[r, hookrightarrow] &\cO_{LE,S(E')} 
        \end{tikzcd}
    \end{center}
    commutes. As such, the tuple of maps $(\iota_E)_{E \in \Omega(\cK/K)}$ defines a homomorphism of $\cR_{L\cK/\QQ}$-modules
    \begin{align*}
        \ES_1(T,\cK/\QQ) \to \ES_1(\ZZ_p(1), L\cK/\QQ)
    \end{align*}
    splitting the natural map $\ES_1(\ZZ_p(1), L\cK/\QQ) \to \ES_1(T,\cK/\QQ)$.

    On the other hand, the group $\cH_E := \Gal(\cG_{LE}/\cG_{LE^p})$ has order prime-to-$p$ by definition. In particular, we have a canonical idempotent $e_{\cH_E} := |\cH_E|^{-1} \sum_{\sigma \in \cH_E} \sigma$ in $\ZZ_p[\cG_{LE}]$ which induces a canonical decomposition
    \begin{align}
        \ZZ_p[\cG_{LE}] &\cong e_{\cH_E}(\ZZ_p[\cG_{LE}]) \oplus (1-e_{\cH_E})(\ZZ_p[\cG_{LE}])\notag\\
        &\cong \ZZ_p[\cG_{LE^p}] \oplus (1-e_{\cH_E})(\ZZ_p[\cG_{LE}])\label{group-ring-idempotent-decomposition}
    \end{align}
    A straightforward calculation using this decomposition then shows that for each pair of fields $E \subseteq E'$ we have a commutative diagram
    \begin{equation}\label{group-ring-norm-diagram}
        \begin{tikzcd}[row sep=3em]
            \ZZ_p[\cG_{LE'}] \arrow[r] \arrow[d, "N_{LE'/LE}"] &\ZZ_p[\cG_{LE^{',p}}] \arrow[d, "N_{LE^{',p}/LE^p}"]\\
            \ZZ_p[\cG_{LE}] \arrow[r] &\ZZ_p[\cG_{LE^{p}}]
        \end{tikzcd}
    \end{equation}
    By tensoring the decomposition (\ref{group-ring-idempotent-decomposition}) over $\ZZ_p[\cG_{LE}]$ with $\cO_{LE}^\times$ we then obtain a decomposition
    \begin{align*}
        \pi_{E}: \cO_{LE}^\times \cong \cO_{LE^p}^\times \oplus (1-e_{\cH_{E^p}})\cO_{LE}^\times
    \end{align*}
    This then combines with the diagram (\ref{group-ring-norm-diagram}) to imply that the tuple of homomorphisms $(\pi_E)_{E \in \Omega(\cL/\QQ)}$ induces a homomorphism of $\cR_{L\cL/\QQ}$-modules
    \begin{align*}
        \ES_1(\ZZ_p(1), L\cK/\QQ) \to \ES_1(\ZZ_p(1), L\cL/\QQ)
    \end{align*}
    which splits the natural map $\ES_1(\ZZ_p(1), L\cL/\QQ) \to \ES_1(\ZZ_p(1), L\cK/\QQ)$.

    The above discussion therefore implies the existence of a natural decomposition
    \begin{align*}
        \ES_1(\ZZ_p(1), L\cL/\QQ) = \ES_1(T, \cK/\QQ) \oplus \ker(\Psi)
    \end{align*}
    Via this we may view the system $\eta \in \ES_1(T,\cK)$ as an element $\overline{\eta} \in \ES_1(\ZZ_p(1), L\cL/\QQ)$ in such a way that $\overline{\eta}_{LM^p} = \eta_{M^p}$ for all $M \in \Omega(\cL/\QQ)$.
    
    It now remains to show that the system $\overline{\eta} \in \ES_1(\ZZ_p(1), L\cL/\QQ)$ gives rise to an Euler-Kolyvagin system of the desired form. To do this we write $\mu_\infty^*$ for the union of all $\mu_m$ with $m$ divisible by $p$ and prime to the conductor $\mathfrak{f}_L$ of $L$. We define a function
    \begin{align*}
        f: \mu_\infty^* &\to \overline{L}^\times \widehat{\otimes}_\ZZ \ZZ_p\\
        \zeta_m &\mapsto \left\{\begin{array}{cc}
            \overline{\eta}_{L(m)} &\text{if } m \not\equiv 2 \pmod{4}\\
            (1-\Fr_2^{-1})\overline{\eta}_{L(m)} &\text{if } m = 2m' \text{ with } m' \geq 1 \text{ odd}
        \end{array}\right.
    \end{align*}
    We note that $f$ is necessarily $\Gal(\overline{L}/L)$-equivariant and satisfies the distribution relation
    \begin{align}
        \prod_{\zeta^a = \varepsilon} f(\zeta) = f(\varepsilon)\label{p-adic-circular-distribution-relation}
    \end{align}
    for all $a \in \NN$ and $\varepsilon \in \mu_\infty^*$.
    
	Given this construction we now set $F = LE$ and claim that the assignment
	\begin{align*}
		\varepsilon_r = \varepsilon_{\eta,d,r} := N_{L(\mathfrak{f}_Er)/F(r)}f\Big(\zeta_{d}\cdot \prod_{\ell \mid r} \zeta_\ell\Big)
	\end{align*}
	defines the desired $p$-adic Euler-Kolyvagin system for $F$.

	The fact that this collection satisfies $(\ES_1)$ and $(\ES_2)$ is immediate from the assumptions. As for $(\ES_3)$, we calculate

	\begin{align*}
		N_{F(r)/F(r/\ell)}N_{L(\mathfrak{f}_Er)/F(r)} (\varepsilon_r) &= N_{F(r)/F(r/\ell)}N_{L({\mathfrak{f}_Er})/F(r)}f\Big(\zeta_{d}\cdot \prod_{q \mid (r/\ell)} \zeta_q\Big)\\
		&= N_{L({(r/\ell)\mathfrak{f}_E})/F({r/\ell})}N_{L({\mathfrak{f}_Er})/L({\mathfrak{f}_E(r/\ell)})}f\Big(\zeta_{d}\cdot \prod_{q \mid (r/\ell)} \zeta_q\Big)\\
		&= N_{L({\mathfrak{f}_E(r/\ell)})/F({r/\ell})}\prod_{i=1}^{\ell-1} f\Big(\zeta_{d}\zeta_\ell^i \prod_{q\mid (r/\ell)} \zeta_q\Big)\\
		&= N_{L({\mathfrak{f}_E(r/\ell)})/F({r/\ell})}\frac{ f\Big(\Big(\zeta_{d} \prod_{q \mid (r/ \ell)} \zeta_q\Big)^\ell\Big)}{ f\Big(\zeta_{d} \prod_{q \mid (r/ \ell)} \zeta_q\Big)}\\
		&= (1-\Fr_\ell^{-1})\cdot N_{L({\mathfrak{f}_E(r/\ell)})/F({r/\ell})}f\Big(\zeta_{d} \prod_{q \mid (r/\ell)} \zeta_q\Big)\\
		&= (1-\Fr_\ell^{-1})\cdot \varepsilon_{r/\ell}
	\end{align*}
	Here the second equality follows from standard functoriality properties of norms, the third by the definition of the norm and the fourth by applying the distribution relation (\ref{p-adic-circular-distribution-relation}).

	Turning now to the verification of the condition $(\ES_4)_p$, we first note that the fact that $\eta$ constitutes a $p$-adic Euler system in the sense of \cite[Def. II.1.1]{rubin} combines with a well-known argument (see, for instance, \cite[Cor. IV.8.1, Rem. IV.8.2]{rubin} or \cite[Thm. A.1]{bss3}) to imply that for each prime $\mathfrak{l}$ of $L(\mathfrak{f}_Er)$ lying above $\ell$ the congruence
	\begin{align}
		f(\zeta_{dr}) \equiv f(\zeta_{d(r/\ell)})^{\Fr_\ell^{-1}} \pmod{\mathfrak{l}}_p\label{euler-kolyvagin-congruence}
	\end{align}
	holds. Since for any primes $p \neq q$ we have $\zeta_{pq} = \zeta_p^{\Fr_q^{-1}}\zeta_q^{\Fr_p^{-1}}$, and $\Fr_q$ acts trivially on $\zeta_q$, a straightforward calculation combines with this congruence to yield the further congruence
	\begin{align*}
		f\Big(\zeta_{d}\cdot \prod_{\ell \mid r} \zeta_\ell\Big) \equiv f\Big(\zeta_{d}\cdot \prod_{q \mid (r/\ell)} \zeta_q\Big) \pmod{\mathfrak{l}}_p
	\end{align*}
	The claim now follows immediately from the definition of the system $(\varepsilon_r)_r$.
	
	Finally, we may calculate
	\begin{align*}
	    \varepsilon_{\eta, n,1} = N_{L(\mathfrak{f}_E)/LE}(f(\zeta_d)) = N_{L(\mathfrak{f}_E)/LE}(\overline{\eta}_{L(d)}) = N_{L\QQ(\mathfrak{f}_E)^p/LE}(\overline{\eta}_{L\QQ(d)^p}) = N_{L\QQ(\mathfrak{f}_E)^p/LE}({\eta}_{\QQ(d)^p})
	\end{align*}
	which completes the proof of the Proposition.
\end{proof}

\subsubsection{The proof of Theorem \ref{pre-question-valid-theorem}}\label{rank-one-results-section}

We are now in a position to state and prove the precise version of Theorem \ref{pre-question-valid-theorem}. 

\begin{theorem}\label{question-valid-theorem}
    Question \ref{main-question} has an affirmative answer for the pair $(T,\cK)$.
\end{theorem}

\begin{proof}
    By Theorem \ref{explicit-conditions-for-conj-theorem} and Remark \ref{mu-vanishing-remark} it suffices to show that for every $\eta \in \ES_1(T,\cK)$, $E \in \Omega(\cK'/K)$ and regular prime $\fp$ of $\bLambda_{E,\fp}$ such that $\braket{(\eta_{E_n})_n}_{\bLambda_{E,\fp}}$ is non-trivial the inclusion (\ref{gm-lemma-inclusion}) of Lemma \ref{Gm-lemma} holds.

    By class field theory we are free to assume that $E$ is the maximal $p$-extension inside $\QQ(m)$ where $m$ is a natural number not divisible by the finite primes in $S$ and not congruent to 2 modulo 4.
    
    To verify the relevant inclusion we first claim that 
    \begin{align*}
        \Fitt^0_{\bLambda_{E,\fp}}\left(\frac{\braket{(\eta_{E_{\chi,n}})_n}}{\braket{(\eta_{E_n})_n)}}\right) = \Fitt_{\bLambda_{E,\fp}}^0(Y_{E,S(E)\setminus S_p(K), \fp}^{{\infty}}) = \left(\prod_{v \in S(E)\setminus S(E_\chi)} 1-\Fr_v^{-1}\right)\bLambda_{E,\fp}
    \end{align*}
    where $E_\chi$ is the maximal $p$-extension inside $\QQ(d)$ where $d$ is such that the character $\chi = \chi_\fp$ corresponding to $\fp$ has conductor either $d$ or $dp$. Indeed, first note that for any prime $q$ dividing the conductor of $E$ we have isomorphisms of $\bLambda_{E}$-modules
    \begin{align*}
        Y_{E,\set{q}}^{{\infty}} \cong \Ind_{D_q}^{\Gal(E_\infty/K)}(\cR) \cong \cR[\Gal(E_\infty/K)/D_q]
    \end{align*}
    where $D_q$ is the decomposition subgroup of $q$ inside $\Gal(E_\infty/K)$. Now, if $I_q$ denotes the inertia subgroup inside $D_q$ then $q$ divides the conductor of $E_\chi$ if and only if $\chi|_{I_q} \neq 1$ whence $Y_{E, \set{q}, \fp}^{{\infty}} = 0$. Conversely, suppose that $\chi|_{I_q} = 1$. We have an exact sequence
    \begin{align*}
        0 \to \bLambda_{E} \xrightarrow{1-\Fr_q^{-1}} \bLambda_E \to \cR[\Gal(E_\infty/K)/D_q] \to 0
    \end{align*}
    Localising at $\fp$ we then obtain
    \begin{align*}
        0 \to \bLambda_{E,\fp} \xrightarrow{1-\Fr_q^{-1}} \bLambda_{E,\fp} \to Y_{E,\set{q}, \fp}^{{\infty}} \to 0
    \end{align*}
    This establishes the latter of the two claimed equalities. As for the first, let $\cH := \Gal(E/E_\chi)$ and observe that the norm element $N_{\cH} = \sum_{\sigma \in \cH} \sigma$ is a unit in $\bLambda_{E,\fp}$. We then have by the Euler system relation an equality
    \begin{align*}
        (\eta_{E_n})_n = N_{E/E_\chi}^{-1}\cdot \left(\prod_{v \in S(E)\setminus S(E_\chi)} 1-\Fr_v^{-1}\right)\cdot (\eta_{E_\chi,n})_n
    \end{align*}
    The claim then follows by combining this equality with the fact that the module $\braket{(\eta_{E_\chi,n})_n}$ is $\bLambda_{E}$-torsion-free.
    
    Given the above discussion we are reduced to verifying that there is an inclusion
    \begin{align}
        \Fitt^0_{\bLambda_{E,\fp}}((U_{E}^\infty/(\eta_{E_{\chi,n}})_n)_\fp) \subseteq \Fitt^0_{\bLambda_{E,\fp}}(A_\fp^\infty)\label{Gm-rank-one-inclusion}
    \end{align}
    Since the localisation map $\bLambda_E \to \bLambda_{E,\fp}$ factors through the map $\bLambda_E \to \bLambda_{E,\chi}$ induced by $\chi$ we are further reduced to verifying the divisibility of characteristic ideals
    \begin{align}\label{final-divisibility}
        \mathrm{char}_{\bLambda_{E,\chi}}\left(A_E^{\infty,\chi}\right) \large\mid \mathrm{char}_{\bLambda_{E,\chi}}\left(\frac{U_E^{\infty,\chi}}{(\eta_{E_\chi,n})_n^\chi}\right)
    \end{align}

   At this stage we seek to apply the argument of Greither used to prove the main conjecture \cite[Th. 3.1]{greither} for abelian fields.
 
    To do this we resume the notations of \S\ref{interlude-section} and fix $n \in \NN$. Then by Proposition \ref{euler-kolyvagin-proposition} there exists a $p$-adic Euler-Kolyvagin $(\varepsilon_{\eta, dp^n, r})_{r \in \cJ}$ system for the field $LE_n$ with the property that
    \begin{align*}
        \varepsilon_1 = \eta_{\QQ(dp^n)^p} = \eta_{E_{\chi,n}}
    \end{align*}
    
    This system may fail to be an Euler system in the strict sense considered by Rubin in \cite{rubin-appendix} (and thus Greither in \cite{greither}). In particular, an Euler system in \textit{loc. cit.} is required to satisfy the property $(\ES_4)$ which is, \textit{a priori}, stronger than $(\ES_4)_p$ where the congruence is required to hold without projecting to the maximal $p$-primary component of the appropriate residue field.
    
    However, the only occurrence of the condition $(\ES_4)$ during the course of the proof of \cite[Th. 3.1]{greither} is in the proof of \cite[Prop. 2.4]{rubin-appendix} which uses the map $\varphi_\ell$ constructed in \cite[Lem. 2.3]{rubin-appendix}. Since $M$ is taken to be a power of $p$, the map $\varphi_\ell$ factors through the projection of $(\cO_{F}/\ell\cO_{F})^\times$ to its $p$-primary component. As such, it suffices to replace $(\ES_4)$ in the proof of \cite[Th. 3.1]{greither} by the weaker condition $(\ES_4)_p$.
    
    Given these observations, we may now apply the argument used to prove \cite[Th. 3.1]{greither} to the aforementioned $p$-adic Euler-Kolyvagin systems arising from the Euler system $\eta$ in order to deduce the desired divisibility (\ref{final-divisibility}).
\end{proof}

\begin{remark}
    It is natural to ask if one can remove the assumption that $p$ does not split completely in $L$. As observed above, this condition ensures that $(\ES_2)$ is always satisfied for the classical $p$-adic Euler systems arising from elements of $\ES_1(T,\cK)$.
    
    Define the module $\ES_{1}^{\mathrm{str}}(T,\cK)$ of `strict' Euler systems to be the subset of $\ES_1(T,\cK)$ comprising those systems $\eta$ with the property that for every $E \in \Omega(\cK/K)$ one has that $\eta_E \in U_E$. Then the proof of Theorem \ref{question-valid-theorem} makes it clear that if $p$ is allowed to split completely in $L$ then one at least has the equality $\ES_1^{\mathrm{str}}(T,\cK) = \ES^b(T,\cK)$ (the backwards inclusion here follows from the argument of Corollary \ref{question-valid-corollary} below and the fact that $\eta^\mathrm{cyc} \in \ES_1^\mathrm{str}(T,\cK)$).
    
    As such, the validity of Question \ref{main-question} in this situation is reduced to the question of whether the equality $\ES_1(T,\cK) = \ES_1^{\mathrm{str}}(T,\cK)$ holds. In this regard we note that by Proposition \ref{evidence-theorem} the condition $\eta_E \in U_E$ holds for any field $E$ satisfying standard hypotheses for the representation $T$. Given this we believe it reasonable to conjecture that this equality always holds.
\end{remark}

We can now state and prove the precise version of Corollary \ref{pre-question-valid-corollary}. To do this we write $\eta^\cyc \in \ES_1(\ZZ_p(1),L\cL/\QQ)$ for the usual Euler system of cyclotomic units. That is to say, the system for which at each $E \in \Omega(L\cL/\QQ)$ one has $\eta_E = N_{\QQ(\mathfrak{f}_E)/E}(1-\zeta_{\mathfrak{f}_E})$. We also continue to use the notation established in the proofs of Proposition \ref{euler-kolyvagin-proposition} and Theorem \ref{question-valid-theorem}.

\begin{corollary}\label{question-valid-corollary}
    $\ES_1(T,\cK)$ is a free $\cR[[\Gal(\cK/K)]]$-module of rank one, with a basis given by $c = \Psi(\eta^\cyc)$.
\end{corollary}

\begin{proof}
    At the outset we fix a field $E \in \Omega(\cK'/\QQ)$ and write $C_{E_n}$ for the group of  cyclotomic units of the field $LE_n$. Explicitly, this is given by the intersection
    \begin{align*}
        C_{E_n} := \set{N_{\QQ(\mathfrak{f}_{LE_n})/LE_n}(1-\zeta_d) | d \text{ divides } \mathfrak{f}_{LE_n}}_{\ZZ_p[\cG_{LE_n}]} \cap (\cO_{LE_n}^\times \otimes_\ZZ \ZZ_p)
    \end{align*}
    We then write $C_{E}^\infty$ for the inverse limit $\varprojlim_{n \in \NN} C_{E_n}^\phi$ taken with respect to the norm maps.
    
    Turning now to the proof of the claim, we note that by Theorem \ref{explicit-conditions-for-conj-theorem} it suffices to show, firstly, that the $\mu_\Lambda$-invariant of the quotient
    \begin{align*}
        \frac{U_{E,S(E)}^\infty}{\braket{(c_{E_n})_n}_{\bLambda_E}}
    \end{align*}
    vanishes and, secondly, that the inclusion (\ref{gm-lemma-inclusion}) is an equality for $\eta = c$.

    The first of these claims follows as an immediate consequence of the main result of Greither in \cite[Appendix]{flach} in combination with the filtration
    \begin{align*}
        \braket{(c_{E_n})_n}_{\bLambda_E} \subseteq C_E^\infty \subseteq U_E^\infty \subseteq U_{E,S(E)}^\infty
    \end{align*}
    As for the second claim, first observe that the proof of Theorem \ref{question-valid-theorem} implies that we just have to verify that (\ref{Gm-rank-one-inclusion}) is an equality in this case.
    We claim, firstly, that the element $c_{E_{\chi,n}}$ generates $C_{\infty, \mathfrak{p}}$. To this end we fix $f \in \NN$ such that $d \mid f \mid m$ and set $\overline{E} = \QQ(f)^p$. Then
    \begin{align*}
        N_{E/E_{\chi}}((c_{\overline{E}_n})_n)= [E:\overline{E}]N_{\overline{E}/E_\chi}((c_{\overline{E}_n})_n) = [E:\overline{E}]\Big(\prod_{v \in S(\overline{E})\setminus S(E_\chi)} 1-\Fr_v^{-1}\Big)((c_{E_{\chi_n}})_n)
    \end{align*}
    Since $N_{E/E_\chi}$ is invertible in $\bLambda_{E,\mathfrak{p}}$ it then follows that $(c_{\overline{E}_n})_n$ is contained in the $\bLambda_{E,\mathfrak{p}}$-module generated by $(c_{E_{\chi_n}})_n$. Now suppose that $d \nmid n$. Then $\chi$ is non-trivial on the group $\Gal(E/\overline{E})$ and so
    \begin{align*}
        ((c_{\overline{E}_n})_n)_\mathfrak{p} = [E:\overline{E}]^{-1}(N_{E/\overline{E}}(c_{\overline{E}_n})_n)_\mathfrak{p} = [E:\overline{E}]^{-1}\chi(N_{E/\overline{E}})((c_{\overline{E}_n})_n)_\mathfrak{p} = 0
    \end{align*}

    The above discussion therefore implies that it suffices to prove the existence of some integer $a$ for which there is an equality of characteristic ideals
    \begin{align*}
        \pi_\chi^a\cdot\mathrm{char}_{\bLambda_{E,\chi}}\left(A_E^{\infty,\chi}\right) = \mathrm{char}_{\bLambda_{E,\chi}}\left(\left(\faktor{U_E^{\infty}}{C_E^\infty}\right)^\chi\right)
    \end{align*}
    where $\pi_\chi$ is a uniformiser of $\QQ(\im(\chi))$. This is now an immediate consequence of Greither's original result \cite[Thm. 3.1]{greither} on the main conjecture.
\end{proof}

\subsubsection{The higher rank setting}

We now focus on the situation where $K$ is an arbitrary totally real field.

We write $\cL = \cL_\psi$ for the maximal subextension of $\cK'$ with the property that every field in $\Omega(\cL K_\infty/K)$ satisfies standard hypotheses for $T$. Note that by Example \ref{standard-hypotheses-example} one knows that the extension $\cL$ is non-trivial.

The exact sequence (\ref{H2Iw-exact-sequence}) implies that for any $E \in \Omega(\cL/K)$, the $\mu$-invariant hypothesis (\ref{mu-invariant-condition}) is equivalent to the vanishing of the $\mu$-invariants of $e_{\chi}\cdot A_E^\infty$ for all $\chi \in I_E$. This has been conjectured to be the case by Iwasawa and we assume this conjecture for all $E \in \Omega(\cL/K)$.

Recall that $T$ is the representation $T = \cR(1) \otimes \psi^{-1}$ defined at the beginning of this section. 

\begin{proposition}\label{divisibilities-valid-higher-rank-theorem}
    Let $\eta \in \ES_r(T,\cK)$ be an Euler system. Then there exists a basic Euler system that agrees with $\eta$ on the extension $\cL K_\infty$.
\end{proposition}

\begin{proof}
    This follows via an immediate application of Corollary \ref{iwasawa-inclusion-holds-corollary}.
\end{proof}

\begin{remark}
    Fix a field $E \in \Omega(\cK/K)$. Then by Example \ref{standard-hypotheses-example} one knows that $E \in \Omega(\cK/K)$ satisfies standard hypotheses if no prime in $S(E)\setminus S_\infty(K)$ splits completely in $L$ (the fixed field of $\psi$). This fact motivates a probabilistic heuristic for the validity of Question \ref{main-question} in this setting.
    
    Indeed, recall that $\cK$ must contain the maximal $p$-subextension of the ray class field modulo $\fq$ for almost every prime $\fq$ of $K$. As such, the analytic density of the set of primes of $K$ ramifying in $\cK$ is one. On the other hand, if we write $|\psi|$ for the order of $\psi$ then by the Chebotarev Density Theorem, the analytic density of the set of primes of $K$ that split completely in $L$ is $[L:K]^{-1} = |\psi|^{-1}$.  In particular the density, in a natural sense, of the set $\Omega(\cL_\psi/K)$ inside $\Omega(\cK'/K)$ is bounded below by $1-|\psi|^{-1}$.
    
    It then follows that as $|\psi|$ approaches infinity, the density of $\Omega(\cL_\psi/K)$ inside $\Omega(\cK'/K)$ approaches one. This in turn implies that as $|\psi|$ approaches infinity, the density of the set of fields in $\Omega(\cK/K)$ upon which a given Euler system agrees with a basic Euler system approaches one.
    
    One can therefore see Question \ref{main-question} in this setting as being `true in the limit $|\psi| \to \infty$'.
\end{remark}

\begin{remark}
    During the preparation of the present manuscript, the article \cite{sakamoto2} of Sakamoto appeared online in which the author also considers the question of to what extent Euler systems are basic. In particular in \cite[Thm. 1.1]{sakamoto2} Sakamoto shows that Question \ref{main-question} has an affirmative answer in the setting of the multiplicative group over totally real fields under the assumed validity of Greenberg's conjecture for every $E \in \Omega(\cK'/K)$.
\end{remark}

\subsection{Elliptic curves and the proof of Theorem \ref{kato-theorem-pre}}
    Let $C/\QQ$ be an elliptic curve and $\QQ_\infty/\QQ$ the cyclotomic $\ZZ_p$-extension of $\QQ$. We denote by $T = T_p(C)$ the $p$-adic Tate module of $C$. Suppose that $p > 3$ and
    \begin{enumerate}
        \item[($\mathrm{EH_1}$)]{The image of the representation
        \begin{align*}
            \rho: G_\QQ \to \Aut(T) \cong \mathrm{GL}_2(\ZZ_p)
        \end{align*}
        contains $\mathrm{SL}_2(\ZZ_p)$.}
        \item[($\mathrm{EH_2}$)]{For all $v \in S\setminus S_\infty(\QQ)$ the curve $C(\QQ_v)$ has no points of order $p$.}
        \item[($\mathrm{EH_3}$)]{There exists a finite abelian extension $L/\QQ$, not contained in $\cK$, such that $C(L)[p^\infty] \neq 0$.}
    \end{enumerate}
    By Example \ref{standard-hypotheses-example}, the Hypotheses $(\mathrm{EH_1})$ and $(\mathrm{EH_2})$ imply that $\QQ$ satisfies standard hypotheses for the representation $T$. In this situation one also knows that the weak Leopoldt conjecture holds for all $E \in \Omega(\cK'/\QQ)$ by the result recalled in Remark \ref{weak-leopolodt-remark}.
    
    We can now state and prove the precise version of Theorem \ref{kato-theorem-pre}.
    
    \begin{theorem}\label{kato-theorem}
        Assume that $p > 3$ and assume to be given an elliptic curve $C/\QQ$ that satisfies the hypotheses $\mathrm{(EH_1)}$ through $\mathrm{(EH_3)}$. Then for every $\eta \in \ES_1(T,\cK)$ there exists a basic Euler system  that agrees with $\eta$ on the extension $\QQ_\infty/\QQ$.
    \end{theorem}
    
    \begin{proof}
        In this situation $r = 1$ and so the Theorem will follow from an immediate application of Corollary \ref{iwasawa-inclusion-holds-corollary} once we verify that the $\mu$-invariant condition (\ref{mu-invariant-condition}) holds.
        
        To do this, first note that vanishing of the $\mu$-invariant of $H^2_\Iw(\cO_{E,S(E)}, T)$ holds if and only if the same is true of that of the dual of the fine Selmer group over $E_\infty$ for any abelian extension $E/\QQ$ by the facts recalled in Remark \ref{mu-vanishing-remark}(iii).
        
        As such, the $\mu$-invariant of $H^2_{\Iw}(\cO_{L,S(L)}, T)$ vanishes by the assumed validity of $(\mathrm{EH}_3)$ and \cite[Cor. 3.6]{coates-sujatha}. It is then straightforward to see that this implies that the $\mu$-invariant of $H^2_\Iw(\cO_{\QQ,S(\QQ)}, T)$ also vanishes as is required to complete the proof of the Theorem.
    \end{proof}
    
    \begin{remark}
        Fouquet and Wan have recently announced a proof of the full Kato Main Conjecture under the same `big image' hypothesis as Kato's divisibility proof in \cite{kato3}. Let $\mathfrak{z}_\infty$ be the element of $Q(\bLambda_\QQ) \otimes_{\bLambda_\QQ} \Det_{\bLambda_\QQ}(C_{\QQ_\infty, S(\QQ)}(T))$ defined in \cite[\S7.1]{bkse}. Since the big image hypothesis is implied by $\mathrm{(EH_1)}$, their result would imply that $\mathfrak{z}_\infty$ is a $\bLambda_\QQ$-basis of $\Det_{\bLambda_\QQ}(C_{\QQ_\infty, S(\QQ)}(T))$. A strengthening of the argument of Corollary \ref{iwasawa-inclusion-holds-proposition} then yields a basic Euler system $c$ that agrees with Kato's Euler system $z^{\mathrm{kato}}$ on the extension $\QQ_\infty/\QQ$ and has the property that for every $E \in \Omega(\cK/K)$ and character $\chi \in \widehat{\cG_E}$ one has that $c_E^\chi \neq 1$ if and only if $(z^\mathrm{kato}_E)^\chi \neq 1$.
    \end{remark}

\appendix
\section*{Appendices}
\renewcommand{\thesubsection}{\Alph{subsection}}

\subsection{Results on \'etale cohomology complexes}\label{cohomology-appendix}

For the convenience of the reader we gather in this appendix various well-known facts about the \'etale cohomology complexes that we have employed in this article. To do this we first recall the definition of the so-called `$\Sigma$-modified' cohomology complexes. We remark that $\Sigma$-modification is only relevant for Appendix \ref{weak-pairs-appendix} and so for the purpose of reading the main body of this article the reader is invited to ignore any adornments by $\Sigma$.

Fix $E \in \Omega(\cK/K)$ and suppose to be given a finite set $U$ of places of $K$ containing $S(E)$ along with a finite set of places $\Sigma$ of $K$ disjoint from $U$. We define the `$\Sigma$-modifed' cohomology complex $\R\Gamma_\Sigma(\cO_{E,U},T)$ to be the mapping fibre in $\Der(\cR[\cG_E])$ of the natural localisation morphism
\begin{align}\label{sigma-modified-definition}
    \R\Gamma(\cO_{E,U}, T) \to \bigoplus_{w \in \Sigma_F} \R\Gamma_f(E_w, T)
\end{align}
Given this, one then defines the `$\Sigma$-modifed' compactly supported cohomology complex
\begin{align*}
    \R\Gamma_{c, \Sigma}(\cO_{E,U}, T) := \R\Hom_{\cR}(\R\Gamma_{\Sigma}(\cO_{E,U}, T^*(1)), \cR)[-3] \oplus \left(\bigoplus_{w \in S_\infty(E)} H^0(E_w, T)\right)[-1]
\end{align*}
To ease notation in the sequel, we then denote
\begin{align*}
    C_{E,U, \Sigma}(T) := \R\Hom_{\cR}(\R\Gamma_{c,\Sigma}(\cO_{E,U}, T^*(1)), \cR)[-2]
\end{align*}
The properties of use to us in this article are then summarised in the following Proposition:

\begin{aproposition}Fix a field $E \in \Omega(\cK/K)$, $U$ a finite set of places of $K$ containing $S(E)$ and $\Sigma$ a finite set of places of $K$ disjoint from $U$. If $H^0(E,T) = 0$ then
    \begin{enumerate}
        \item{$C_{E,U,\Sigma}(T)$ is a perfect complex of $\cR[\cG_E]$-modules and is acyclic outside degrees zero and one. Moreover, the Euler-Poincar\'e characteristic of the complex $Q \otimes_R C_{E,U,\Sigma}(T)$ vanishes.}
        \item{There is a canonical isomorphism $H^0(C_{E,U,\Sigma}(T)) \cong H^1_{\Sigma}(\cO_{E,U}, T)$ and a split exact sequence
            \begin{align}
                0 \to H^2_\Sigma(\cO_{E,U}, T) \to H^1(C_{E,U,\Sigma}(T)) \to Y_K(T)^* \otimes_{\cR} \cR[\cG_E] \to 0\label{complex-cohomology-sequence}
            \end{align}
            in which the injection is canonical and the surjection is dependent on a choice of a set of representatives of the $\Gal(\cK/K)$-orbits on $S_\infty(\cK)$.
        }
        \item{For any finite set of places $\Sigma'$ of $K$ containing $\Sigma$ and disjoint from $U$ one has an exact triangle
            \begin{align}
                C_{E,U,\Sigma'}(T) \to C_{E,U, \Sigma}(T) \to \bigoplus_{w \in (\Sigma'\setminus \Sigma)_E} \R\Gamma_f(E_w, T)[1]\label{change-sigma-triangle}
            \end{align}
        in $\Der^p(\cR[\cG_E])$.
        }
        \item{For any set of places $U'$ of $K$ containing $U$ and disjoint from $\Sigma$ one has an exact triangle
            \begin{align}
                C_{E,U,\Sigma}(T) \to C_{E,U', \Sigma}(T) \to \bigoplus_{w \in (U'\setminus U)_E} \R\Gamma_f(E_w, T^*(1))^*[-1]\label{change-U-triangle}
            \end{align}
        in $\Der^p(\cR[\cG_E])$.
        }
        \item{For any pair of fields $E \subseteq E'$ in $\Omega(\cK/K)$ such that $U$ contains $S(E')$, there is a natural codescent isomorphism
            \begin{align}
                C_{E', U, \Sigma}(T) \otimes_{\cR[\cG_{E'}]} \cR[\cG_E] \cong C_{E,U, \Sigma}(T)\label{codescent-iso}
            \end{align}
        in $\Der^p(\cR[\cG_E])$.
        }
    \end{enumerate}

    \begin{proof}
        Part (1) and (2) of the Proposition are exactly \cite[Prop. 2.21]{bs}. Part (4) is Subsection 2.3.2 of \textit{loc. cit.} and Part (5) is \cite[Prop. 1.6.5(iii)]{fukaya-kato}. As for Part (3), the desired exact triangle follows from a straightforward calculation involving combining Artin-Verdier duality (and noting that $p$ is odd) in the form \cite[(6)]{bf} with  the triangle of definition of the $\Sigma$-modified cohomology complex for $T$.
    \end{proof}
\end{aproposition}

\subsection{Weakly admissible pairs}\label{weak-pairs-appendix}

It is natural to ask if one can weaken the hypotheses levied on the pairs $(T,\cK)$ in order to say something about more general representations. For example, if $K$ is an imaginary quadratic field then hypothesis $\mathrm{(H_3)}$ is never satisfied for the representation $\ZZ_p$ if $\cK$ contains the cyclotomic $\ZZ_p$-extension of $K$ and so the Euler system of elliptic units defined over $K$ do not fit into the theoretical framework underpinning Question \ref{main-question}. In this section we outline how one can weaken the hypotheses on $(T,\cK)$ to incorporate such examples of representations.

Consider the following hypothesis on pairs $(T,\cK)$:

\begin{ahypothesis}\label{weak-T-hypothesis}\text{}
    \begin{enumerate}
        \item[$\mathrm{(H_0)}$]{The $\cR$-module $Y_K(T) := \bigoplus_{v \in S_\infty(K)} H^0(K_v,T^*(1))$ is free of rank $ r \geq 1$.}
        \item[$(\mathrm{H_1})$]{$T$ is ramified at only finitely many primes of $K$.}
        \item[$(\mathrm{H_2})$]{For all $E \in \Omega(\cK/K)$ the invariants module $H^0(E, T)$ vanishes.}
        \item[$(\mathrm{H_3'})$]{There exists a finite set of places $S_f$ of $K$ with the property that for every $v \not\in S_f$, the complex $\R\Gamma_f(K_v, T)$ is acyclic outside degree one and $H^1_f(K_v, T)$ is a cyclic $\cR$-module.}
    \end{enumerate}
\end{ahypothesis}

Throughout this appendix we assume that all pairs $(T,\cK)$ satisfy Hypotheses \ref{weak-T-hypothesis} and \ref{k-hypothesis} and that $S$ contains $S_f$. In this case we shall refer to such pairs as `weakly admissible pairs'.

\begin{aexample}
   Set $T = \ZZ_p(1)$ and fix a field $E \in \Omega(\cK/K)$. In this case the acyclicity condition in $\mathrm{(H_3')}$ is clear from the definition of $T$. Moreover, since for every $v \not\in S(E)$, $\R\Gamma_f(K_v, T)$ is isomorphic in $\Der^p(\ZZ_p[\cG_E])$ to the complex $T \xrightarrow{1-\Fr_v^{-1}} T$, and $T$ is a free $\ZZ_p$-module of rank one, it follows that $H^1_f(K_v, T)$ is cyclic. As such, $(T,\cK)$ is a weakly admissible pair for any appropriate choice of $\cK$.
\end{aexample}

\subsubsection{The ideal $\cA_{T,\cK}$ and the modified question}

Observe that since $T$ does not in general satisfy $\mathrm{(H_3)}$ the homomorphism $\Theta_{T,\cK}$ constructed by Burns and Sano is only a map
\begin{align*}
    \Theta_{T,\cK}: \VS(T,\cK) \to \RES_r(T,\cK)
\end{align*}
We refer to the image of this map as the module of `basic Euler systems' for the weakly admissible pair $(T,\cK)$ and denote it as usual by $\ES^b(T,\cK)$. In this section we shall construct a non-zero ideal $\cA_{T,\cK}$ of $\cR[[\Gal(\cK/K)]]$ with the property that $\cA_{T,\cK}\cdot \ES^b(T,\cK) \subseteq \ES_r(T,\cK)$.

To this end, fix a field $E \in \Omega(\cK/K)$ and define the ideal $\cA_{T,E}$ of $\cR[\cG_E]$ to be
\begin{align*}
    \cA_{T,E} := \Ann_{\cR[\cG_E]}(H^1(\cO_{E,S(E)}, T)_\tor)
\end{align*}
where here $\tor$ indicates the $R$-torsion submodule. We now record the following Lemma which can be seen as an analogue of \cite[Lem. IV.1.1]{tate} for $p$-adic representations:

\begin{alemma}\label{tate-generalised-lemma}
    For each $E \in \Omega(\cK/K)$ the following claims are valid:
    \begin{enumerate}
        \item{Let $U$ be a finite set of places of $K$ not containing $S(E)$. Then there is an equality
            \begin{align*}
                \cA_{T,E} = \Braket{\Ann_{\cR[\cG_E]}(H^1_f(K_v, T_E)) \mid v \not\in U}
            \end{align*}
        }
        \item{For every field $E' \in \Omega(\cK/K)$ containing $E$ the natural projection map $\pi_{E'/E}: \cR[\cG_{E'}] \to \cR[\cG_E]$ sends $\cA_{T,E'}$ onto $\cA_{T,E}$.}
    \end{enumerate}
\end{alemma}

\begin{proof}
    At the outset we note that in order to prove the first claim it suffices, by Shapiro's Lemma, to demonstrate that
    \begin{align}
        \Ann_{\cR[\cG_E]}(H^1(\cO_{K,S(E)}, T_E)_\tor) = \Braket{\Ann_{\cR[\cG_E]}(H^1_f(K_v, T_E)) \mid v \not\in U}\label{desired-annihilator-equality}
    \end{align}
    Moreover, an analysis of the long exact sequence of cohomology of the triangle (\ref{change-U-triangle}) implies that there is an equality $H^1(\cO_{K,U}, T_E)_\tor = H^1(\cO_{K,S(E)}, T_E)_\tor$. As such, we may verify the above equality with the set $S(E)$ replaced by $U$.

   To simplify notation in this proof, we shall write $V = T_E \otimes_R Q$, $K_U$ for the the maximal Galois extension of $K$ unramified outside $U$ and $G_{K,U}$ for its Galois group. The long exact sequence of $G_{K,U}$-cohomology of the tautological exact sequence
    \begin{align}
        0 \to T_E \to V \to V/T_E \to 0\label{tautological-sequence}
    \end{align}
    combines with the assumed validity of $\mathrm{(H_3')}$ to give an identification
    \begin{align*}
        H^1(\cO_{K,U}, T_E)_\tor \cong H^0(\cO_{K,U}, V/T_E) = (V/T_E)^{G_{K,U}}
    \end{align*}
    For each $v \not\in U$, let $\mathrm{Frob}_v$ be the associated Frobenius conjugacy class in $G_{K,U}$. Then by the Chebotarev density theorem applied to the extension $K_{U}/K$, one knows that the set $\bigcup_{v \not\in U} \mathrm{Frob}_v$ is dense in $G_{K,U}$. As such, if we are given a $\sigma \in G_{K,U}$, there exists a place $v \not\in U$  with the property that $\sigma \in \overline{\braket{\mathrm{Frob}_v}}$. But the latter is just the decomposition group $G_{K,v}$ relative to $v$ inside $G_{K,U}$. We thus see that
    \begin{align*}
        (V/T_E)^{G_{K,U}} = \bigcap_{v \not\in U} (V/T_E)^{G_{K,v}} = \bigcap_{v \not\in U} H^0(K_v, V/T_E)
    \end{align*}
    Now for each $v \not\in U$, the long exact sequence of $\Gal(K_v^\ur/K_v)$-cohomology of (\ref{tautological-sequence}) together with the assumed validity of $\mathrm{(H_3')}$ implies that there is a canonical identification
    \begin{align*}
        H^0(K_v, V/T_E) = H^0_f(K_v, V/T_E) \cong H^1_f(K_v, T_E)
    \end{align*}
    The desired equality (\ref{desired-annihilator-equality}) now follows via passage to annihilators after combining the above two facts.

    Turning now to the second claim, we fix a finite set of places $U$ of $K$ disjoint from $S(E')$ and choose $v \in U$. By using the resolution (\ref{finite-support-resolution}) one deduces the existence of a codescent isomorphism
    \begin{align*}
        H^1_f(K_v, T_{E'}) \otimes_{\cR[\cG_{E'}]} \cR[\cG_E] \cong H^1_f(K_v, T_E)
    \end{align*}
    By Hypothesis $\mathrm{(H_3')}$, the $\cR[\cG_{E'}]$-module $H^1_f(K_v, T_E)$ is of projective dimension one. In particular, its initial Fitting ideal is an invertible $\cR[\cG_E]$-module. We may therefore take Fitting ideals across the above isomorphism to deduce that the restriction map $\cR[\cG_{E'}] \to \cR[\cG_E]$ induces a surjective map
    \begin{align*}
        \Fitt_{\cR[\cG_{E'}]}^0(H^1_f(K_v, T_{E'})) \to \Fitt_{\cR[\cG_E]}^0(H^1_f(K_v, T_E))
    \end{align*}
    On the other hand, the assumed validity of $\mathrm{(H_3')}$ implies that for every $v \not\in U$, $H^1_f(K_v, T_E)$ is a cyclic $\cR[\cG_E]$-module. As such its $\cR[\cG_E]$-annihilator coincides with its initial Fitting ideal. The Lemma now follows by applying the first assertion of the Lemma with the set $U$ to $\cA_{T,E}$.
\end{proof}

We can now define an ideal of $\cR[[\Gal(\cK/K)]]$
\begin{align*}
    \cA_{T,\cK} := \varprojlim_{E \in \Omega(\cK/K)} \cA_{T,E}
\end{align*}
where the inverse limit is taken with respect to the maps $\pi_{E'/E}$. A straight-forward exercise in class field theory (see, for example, \cite[Lem. 3.5]{bdss}) implies that $\cA_{T,\cK}$ is non-zero.

\begin{aproposition}
    There is an inclusion $\cA_{T,\cK}\cdot\ES^{b}(T,\cK) \subseteq \ES_r(T,\cK)$.
\end{aproposition}

\begin{proof}
    Fix an element $a = (a_E)_E$ of $\cA_{T,\cK}$, a vertical system $(z_E)_E \in \VS(T,\cK)$ and a field $E \in \Omega(\cK/K)$. By definition, we are required to show that
    \begin{align*}
        a_E \cdot z_E \subseteq \bidual_{\cR[\cG_E]}^r H^1(\cO_{E,S(E)}, T)
    \end{align*}
    By the definition of the ideal $\cA_{T,E}$ and Lemma \ref{v-torsion-free-lemma} below, we may choose a singleton set $\Sigma = \set{v} \not\subseteq S(E)$ with the property that $H^1_{\Sigma}(\cO_{E,S(E)}, T)$ is $R$-torsion-free and assume that $a_E$ is of the form
    \begin{align*}
        a_E \in \Fitt_{\cR[\cG_E]}^0\left(H^1_f(K_v, T_E)\right)
    \end{align*}
    Note now that there are canonical isomorphisms of $\cR[\cG_E]$-modules
    \begin{align*}
        \Det_{\cR[\cG_E]}(C_{E,S(E), \Sigma}(T)) &\cong \Det_{\cR[\cG_E]}\left(\R\Gamma_f(K_v, T_E)\right)\otimes_{\cR[\cG_E]} \Det_{\cR[\cG_E]}(C_{E,S(E)}(T))\\
        &\cong \left(\Fitt_{\cR[\cG_E]}^0H^1_f(K_v, T_E)\right)\cdot \Det_{\cR[\cG_E]}(C_{E,S(E)}(T))
    \end{align*}
    where the first is induced by the exact triangle (\ref{change-sigma-triangle}) and the second by the assumed validity of $\mathrm{(H_3')}$.
    
    Since there is an inclusion $\bidual_{\cR[\cG_E]}^r H^1_{\Sigma}(\cO_{E,S(E)}, T) \subseteq \bidual_{\cR[\cG_E]}^r H^1(\cO_{E,S(E)}, T)$, we are reduced to showing that
    \begin{align*}
        \Theta_{T, E}(\Det_{\cR[\cG_E]}(C_{E,S(E), \Sigma}(T))) \subseteq \bidual_{\cR[\cG_E]}^r H^1_\Sigma(\cO_{E,S(E)}, T)
    \end{align*}
    
    To accomplish this, observe that since $H^1_{\Sigma}(\cO_{E,S(E)}, T)$ is $R$-torsion-free, the complex $C_{E,S(E),\Sigma}(T)$ constitutes an admissible complex of $\cR[\cG_E]$-modules in the sense of \cite[Def. 2.20]{bs}.
    
    The Theorem now follows immediately upon appealing to Prop. A.11(ii) of \textit{loc. cit.} with the data $(\cR, C, X)$ taken to be $(\cR[\cG_E], C_{E,S(E), \Sigma}(T), Y_E(T)^*)$ and the surjection $f$ to be the right-hand map of the exact sequence (\ref{complex-cohomology-sequence}).
\end{proof}

\begin{alemma}\label{v-torsion-free-lemma}
    Fix a field $E \in \Omega(\cK/K)$. Then for any singleton set $\Sigma := \set{v} \not\subseteq S(E)$ one has that $H^1_\Sigma(\cO_{E,S(E)}, T)$ is $R$-torsion-free.
\end{alemma}

\begin{proof}
    An analysis of the long exact sequence of cohomology of the triangle (\ref{change-sigma-triangle}) implies that $H^1_\Sigma(\cO_{E,S(E)}, T)_\tor$ identifies with the kernel of the map
    \begin{align*}
        \Delta: H^1(\cO_{E,S(E)}, T)_\tor \to \bigoplus_{w \in \Sigma_E} H^1_f(E_w, T)
    \end{align*}
    Now fix a place $w$ of $E$ lying over $v$ and consider the inflation-restriction exact sequence
    \begin{align*}
        0 \to H^1_f(E_w, T) \to H^1(E_w, T) \to H^1(E_w^\ur, T)^{\Gal(E_w^\ur/E_w)}
    \end{align*}
    Observe that since $T$ has good reduction at $v$, there is an equality
    \begin{align*}
        H^1(E_w^\ur, T) = \Hom_{\cR}(\Gal(E_w^c/E_w^\ur), T)
    \end{align*}
    which is clearly $R$-torsion-free. As such we see that $H^1_f(E_w, T)$ (which is torsion by $\mathrm{(H_3')}$) identifies with $H^1(E_w, T)_\tor$. Write now $V = T \otimes_R Q$. Applying $E_w$-cohomology to the tautological sequence (\ref{tautological-sequence}) gives a canonical identification
    \begin{align*}
        H^1(E_w, T)_\tor \cong H^0(E_w, V/T)
    \end{align*}
    A similar argument also shows that $H^1(\cO_{E,S(E)}, T)_\tor$ identifies with $H^0(\cO_{E,S(E)}, V/T)$. Hence $\ker(\Delta)_\tor$ identifies with the kernel of the natural map
    \begin{align*}
        H^0(\cO_{E,S(E)}, V/T) \to \bigoplus_{w \in \Sigma_E} H^0(E_w, V/T)
    \end{align*}
    which is clearly injective.
\end{proof}

Given this construction we are now in a position to formulate a version of Question \ref{main-question} for weakly admissible pairs which can be seen as a direct analogue of \cite[Conj. 2.5]{bdss}:

\begin{aquestion}\label{weak-main-question}
    Is there an inclusion $\cA_{T,\cK}\cdot \ES_r(T,\cK) \subseteq \cA_{T,\cK}\cdot\ES^{b}(T,\cK)$?
\end{aquestion}

\begin{aremark}
        If $(T,\cK)$ is an admissible pair then, for every $E \in \Omega(\cK/K)$, $\cA_{T,E} = \cR[\cG_E]$ whence $\cA_{T,\cK} = \cR[[\Gal(\cK/K)]]$. Hence in this case Question \ref{weak-main-question} reduces to Question \ref{main-question}.
\end{aremark}

\enlargethispage{2em}

\end{document}